%% file: neuinger5.tex
\documentclass[a4paper,10pt]{amsart}
 \usepackage{amsfonts,mathrsfs}
 \usepackage{amsthm}
 \usepackage{amssymb}
 \usepackage{enumerate}
 \usepackage{tikz-cd}
 \usepackage{cleveref}
 \usepackage{graphicx}
\theoremstyle{definition}
\newtheorem{Def}{Definition}[section]
\newtheorem{ex}[Def]{Example}
\newtheorem{rem}[Def]{Remark}

\theoremstyle{plain}
\newtheorem{prop}[Def]{Proposition}
\newtheorem{thm}[Def]{Theorem}
\newtheorem*{thm*}{Theorem}
\newtheorem{lem}[Def]{Lemma}
\newtheorem{prob}[Def]{Problem}
\newtheorem{cor}[Def]{Corollary}
\newtheorem*{cor*}{Corollary}

\newtheorem*{con*}{Conjecture}
\newtheorem*{frag*}{Question}
\newtheorem*{verm*}{Vermutung}

\include{operators}
\include{fonts}

\title[Positive Ulrich Sheaves]{Positive Ulrich Sheaves}
\author{Christoph Hanselka}
\address{Universit\"at Konstanz, Germany} 
\email{christoph.hanselka@uni-konstanz.de}
\author{Mario Kummer}
\address{Technische Universit\"at Berlin, Germany} 
\email{kummer@tu-berlin.de}
\thanks{The second author has been supported by the DFG under Grant No.421473641.}
\newcommand{\comment}[1]{}

\begin{document}


\begin{abstract}
 We provide a criterion for a coherent sheaf to be an Ulrich sheaf in terms of a certain bilinear form on its global sections. When working over the real numbers we call it a positive Ulrich sheaf if this bilinear form is symmetric or hermitian and positive definite. In that case our result provides a common theoretical framework for several results in real algebraic geometry concerning the existence of algebraic certificates for certain geometric properties. For instance, it implies Hilbert's theorem on nonnegative ternary quartics, via the geometry of del Pezzo surfaces, and the solution of the Lax conjecture on plane hyperbolic curves due to Helton and Vinnikov.
\end{abstract}
\maketitle

\section{Introduction}
A widespread principle in real algebraic geometry is to find and use algebraic certificates for geometric statements.
For example, a \emph{sum of squares} representation of a homogeneous polynomial $p\in\R[x_1,\ldots,x_n]_{2d}$ of degree $2d$ is a finite sequence of polynomials $g_1,\dots,g_r\in\R[x_1,\ldots,x_n]_d$ such that $$p=g_1^2+\ldots+g_r^2$$ and serves as an algebraic certificate of the geometric property that $p$ takes only nonnegative values at real points: $p(x)\geq0$ for all $x\in\R^n$. 
In an influential paper David Hilbert \cite{hilbert1888} showed that the converse is true if (and only if) $2d=2$, or $n=1$, or $(2d,n)=(4,3)$.
While in the first two cases this can be seen quite easily via linear algebra and the fundamental theorem of algebra respectively, the proof for the 
case $(2d,n)=(4,3)$ of ternary quartics is nontrivial. 
There have been several different new proofs of this statement in the last twenty years: via the Jacobian of the plane curve defined by $p$ \cite{jakobsottile} relying on results of Arthur Coble \cite{coble}, using elementary techniques \cite{clauspfister}, and as a special case of more general results on varieties of minimal degree \cite{minimalgreg}. 

Another instance, that has attracted a lot of attention recently, is the following. Let $h\in\R[x_0,\ldots,x_n]_d$ be a homogeneous polynomial for which there are real symmetric (or complex hermitian) matrices $A_0,\ldots,A_n$ and $r\in\Z_{>0}$ such that
$$h^r=\det(x_0A_0+\ldots+x_nA_n)$$
and $e_0A_0+\ldots+e_nA_n$ is positive definite for some $e\in\R^{n+1}$. In this case, we say that $h^r$ has a \emph{definite symmetric (or hermitian) determinantal representation} and it is a certificate that 
$h$ is \emph{hyperbolic with respect to $e$} in the sense that the univariate polynomial $h(te-v)\in\R[t]$ has only real zeros for all $v\in\R^{n+1}$: The minimal polynomial of a hermitian matrix has only real zeros. Peter Lax \cite{lax58} conjectured that for $n=2$ and arbitrary $d\in\Z_{>0}$, the following strong converse is true: Every hyperbolic polynomial $h\in\R[x_0,x_1,x_2]_d$ has a definite symmetric determinantal representation (up to multiplication with a nonzero scalar). This conjecture was solved to the affirmative by Bill Helton and Victor Vinnikov \cite{HV07}. Furthermore, every hyperbolic $h\in\R[x_0,x_1,x_2,x_3]_3$ has a definite hermitian determinantal representation \cite{buck} and for every quadratic hyperbolic polynomial  $h\in\R[x_0,\ldots,x_n]_2$ there is a $r\in\Z_{>0}$ such that $h^r$ has definite symmetric (or hermitian) determinantal representation \cite{timundthom}. On the other hand, if $d\geq4$ and $n\geq3$ or if $d\geq3$ and $n\geq42$, there are hyperbolic polynomials $h\in\R[x_0,\ldots,x_n]_d$ such that no power $h^r$ has a definite symmetric (or hermitian) determinantal representation \cite{Bra11, sos42}. The cases $(d,n)$ with $d=3$ and $3<n<42$ are  open.

The condition of $h\in\R[x_0,\ldots,x_n]_d$ being hyperbolic with respect to $e$ can be phrased in geometric terms as follows: Let $X=\cV(h)\subset\pp^n$ be the hypersurface defined by $h$. Then $h$ is hyperbolic with respect to $e$ if and only if the linear projection $\pi_e:X\to\pp^{n-1}$ with center $e$ is \emph{real fibered} in the sense that $\pi_e^{-1}(\pp^{n-1}(\R))\subset X(\R)$. This leads to a natural generalization of hyperbolicity to arbitrary embedded varieties introduced in \cite{sha14} and further studied in \cite{eliich}: A subvariety $X\subset\pp^n$ of dimension $k$ is \emph{hyperbolic} with respect to a linear subspace $E\subset\pp^n$ of codimension $k+1$ if $X\cap E=\emptyset$ and the linear projection $\pi_E:X\to\pp^k$ with center $E$ is real fibered. Furthermore, a polynomial $p\in\R[x_1,\ldots,x_n]_{2d}$ is nonnegative if and only if the double cover $X\to\pp^n$ ramified along the zero set of $p$ is real fibered, where $X$ is defined as the zero set of $y^2-p$ in a suitable weighted projective space. Thus both above mentioned geometric properties of  polynomials, being nonnegative and being hyperbolic, can be seen as special instances of real fibered morphisms.

Let $f:X\to Y$ be a morphism of schemes. A coherent sheaf $\cF$ on $X$ is called $f$\emph{-Ulrich} if there is a natural number $r>0$ such that $f_*\cF\cong\cO_Y^r$. If $X\subset\pp^n$ is a closed subscheme, then one defines a coherent sheaf $\cF$ on $X$ to be an \emph{Ulrich sheaf} per se if it is $\pi$-Ulrich for any finite surjective linear projection $\pi:X\to\pp^k$. This is equivalent to $H^i(X,\cF(-j))=0$ for $1\leq j\leq \dim(X)$ and all $i$. See \cite{ulrichintro} for an introduction to Ulrich sheaves. The question of which subvarieties of $\pp^n$ carry an Ulrich sheaf is of particular interest in the context of Boij--S\"oderberg theory \cite{boij}. 

For real fibered morphisms $f:X\to Y$ \emph{positive} $f$-Ulrich sheaves have been defined in \cite{eliich} and it was shown that a hypersurface in $\pp^n$ carries a positive Ulrich sheaf if and only if it is cut out by a polynomial with a definite determinantal representation. For subvarieties $X\subset\pp^n$ of higher codimension supporting a positive Ulrich sheaf is equivalent to admitting some generalized type of determinantal representation that was introduced in \cite{sha14} and motivated by operator theory. The existence of such a determinantal representation for $X$ implies that some power of the Chow form of $X$ has a definite determinantal representation. On the other hand, we will show that a real fibered double cover $f:X\to\pp^n$ ramified along the zero set of a homogeneous polynomial $p$ admits a positive $f$-Ulrich sheaf if and only if $p$ is a sum of squares. Therefore, the notion of positive Ulrich sheaves encapsulates both types of algebraic certificates for the above mentioned geometric properties of homogeneous polynomials, namely being a sum of squares and having a definite determinantal representation.

The main result of this article is a criterion for a coherent sheaf to be a positive Ulrich sheaf which implies all the above mentioned existence results on representations as a sum of squares and determinantal representations, and more. It only comprises a positivity criterion that can be checked locally as well as a condition on the dimension of the space of global sections (but suprisingly not of the higher cohomology groups) of the coherent sheaf at hand. 

To this end, after some preparations in \Cref{sec:symm}, we characterize Ulrich sheaves in terms of a certain bilinear mapping and its behavior on the level of global sections in \Cref{sec:ulr}, see in particular \Cref{thm:ulrichifnondeg}. In this part we work over an arbitrary ground field and we believe that these results can be of independent interest. 

After we review some facts on the codifferent sheaf in \Cref{sec:codif} that will be important later on, we focus on varieties over $\R$. In \Cref{sec:rf} we recall some facts about {real fibered morphisms}. We then define positive Ulrich sheaves in \Cref{sec:psd} after some preparations in \Cref{sec:herm}. \Cref{thm:ulrichifpsd} is the above mentioned convenient criterion for checking whether a sheaf is a positive Ulrich sheaf.
In \Cref{sec:detreps} we show that
for a given polynomial 
having a determinantal or sum of squares representation is equivalent to the existence of a certain positive Ulrich sheaf. From this the result on determinantal representations of quadratic hyperbolic polynomials from \cite{timundthom} follows directly. In order to make our general theory also applicable to other cases of interest, we specialize to Ulrich sheaves of rank one on irreducible varieties in \Cref{sec:smooth}. \Cref{thm:weil} gives a convenient criterion for  a Weil divisor giving rise to a positive Ulrich sheaf. Namely, under some mild assumptions, if $f:X\to Y$ is a real fibered morphism and $D$ a Weil divisor on $X$ such that $2D$ differs from the ramification divisor of $f$ only by a principal divisor defined by a nonnegative rational function, then the sheaf associated to $D$ is a positive $f$-Ulrich sheaf whenever its space of global sections has dimension $\deg(f)$.
A particularly nice case is that of hyperbolic hypersurfaces: The existence of a definite determinantal representation is guaranteed by the existence of a certain \emph{interlacer}, generalizing the construction for plane curves from \cite{hvelemt}, see \Cref{cor:dixon}.

In \Cref{sec:curves} we apply our theory to the case of curves and show how it easily implies the Helton--Vinnikov theorem on plane hyperbolic curves \cite{HV07} as well as its generalization to curves of higher codimension from \cite{sha14} using the $2$-divisibility of the Jacobian. Finally, in \Cref{sec:delpezzo} we consider the anticanonical map on real del Pezzo surfaces in order to reprove Hilbert's theorem on ternary quartics \cite{hilbert1888} and the existence of a hermitian determinantal representation on cubic hyperbolic surfaces \cite{buck}. We further prove a new result on quartic del Pezzo surfaces in $\pp^4$. Apart from our general theory, the only ingredients for this part are basic properties of (real) del Pezzo surfaces as well as the Riemann--Roch theorem.

\section{Preliminaries and notation}
For any scheme $X$ and $p\in X$ we denote by $\kappa(p)$ the residue class field of $X$ at $p$. If $X$ is separated, reduced (but not necessarily irreducible) and of finite type over a field $K$, we say that $X$ is a variety over $K$. For any coherent sheaf $\cF$ on $X$ we denote by $\rank_p(\cF)$ the dimension of the fiber of $\cF$ at $p$ considered as $\kappa(p)$-vector space. If $X$ is irreducible with generic point $\xi$, we simply denote $\rank(\cF)=\rank_\xi(\cF)$. If $X$ is a scheme (over a field $K$) and $L$ a field (extension of $K$), then we denote by $X(L)$ the set of all morphisms $\Spec(L)\to X$ of schemes (over $K$). For a field $K$ we let $\pp^n_K=\textrm{Proj}(K[x_0,\ldots,x_n])$ and if the field is clear from the context we omit the index and just write $\pp^n$. We say that a scheme is \emph{noetherian} if it can be covered by a finite number of open affine subsets that are spectra of noetherian rings.

\section{Bilinear mappings on coherent sheaves}\label{sec:symm}

\begin{Def}
 Let $X$ be a scheme and let $\cF_1,\cF_2$ and $\cG$ be  coherent sheaves on $X$. A $\cG$-\textit{valued pairing} of $\cF_1$ and $\cF_2$ is a morphism of coherent sheaves $\varphi: \cF_1 \otimes \cF_2 \to \cG$. Let $K$ be a field and $\alpha\in X(K)$, i.e., a morphism $\alpha: \Spec (K) \to X$. Then we get a bilinear map $\alpha^* \varphi$ on $\alpha^* \cF_1 \times \alpha^*\cF_2$ with values in $\alpha^* \cG$ which are just finite dimensional $K$-vector spaces. We say that $\varphi$ is \textit{nondegenerate at $\alpha\in X(K)$} if the map $\alpha^* \cF_1 \to \Hom_K( \alpha^*\cF_2, \alpha^* \cG) $ induced by $\alpha^* \varphi$ is an isomorphism.
\end{Def}

For the rest of this section, unless stated otherwise, let $X$ always be a geometrically integral scheme with generic point $\xi$ which is proper over a field $K$.

\begin{lem}\label{lem:basisfree}
 Let $\cF$ be a coherent sheaf on $X$ which is generated by global sections. If there is a $K$-basis of $H^0(X,\cF)$ which is also a $\kappa(\xi)$-basis of $\cF_\xi$, then $\cF\cong\cO_X^r$ where $r=\dim H^0(X,\cF)$.
\end{lem}

\begin{proof}
 Let $\cK$ be the kernel of the map $\cO_X^r\to\cF$ that sends the unit vectors to the $K$-basis of $H^0(X,\cF)$. Since $\cF$ is generated by global sections, this map is surjective. We thus obtain the short exact sequence $$0\to\cK\to\cO_X^r\to\cF\to0.$$ Passing to the stalk at $\xi$ gives $\cK_\xi=0$ by our assumption. Since $\cK$ is torsion-free as a subsheaf of $\cO_X^r$, this implies that $\cK=0$ and therefore $\cO_X^r\cong\cF$.
\end{proof}

\begin{rem}
 Let $\varphi: \cF_1\otimes_{\cO_X} \cF_2 \to \cO_X$ be a pairing of the coherent sheaves $\cF_1$ and $\cF_2$. This induces a bilinear mapping $$V_1 \times V_2\to K$$ where $V_i=H^0(X,\cF_i)$ since $H^0(X,\cO_X)=K$. 
\end{rem}

\begin{lem}\label{lem:linindep}
 Let $\cF_1$ and $\cF_2$ be coherent sheaves on $X$, let $V_i=H^0(X,\cF_i)$ and let $s_1,\ldots,s_r$ be a basis of $V_1$. Let $\varphi: \cF_1\otimes_{\cO_X}\cF_2\to\cO_X$ be a pairing such that the induced  bilinear mapping $V_1 \times V_2\to K$ is nondegenerate. Then the images of $s_1,\ldots,s_r$ in the $\kappa(\xi)$-vector space $(\cF_1)_\xi$ are linearly independent.
\end{lem}

\begin{proof}
 Since the bilinear mapping $V_1 \times V_2\to K$ is nondegenerate, there is a basis $t_1,\ldots,t_r\in V_2$ that is dual to $s_1,\ldots,s_r$ with respect to this bilinear mapping. Suppose $f_1,\ldots,f_r\in\kappa(\xi)$ such that $$f_1s_1+\ldots+f_rs_r=0.$$ Tensoring with $t_j$ and applying $\varphi$ yields $f_j\cdot\varphi(s_j\otimes t_j)=0$ and therefore $f_j=0$ since $\varphi(s_j\otimes s_j)=1$ by assumption.
\end{proof}

\begin{prop}\label{prop:ifglobgen}
 Let $\cF_1$ and $\cF_2$ be coherent sheaves on $X$, let $V_i=H^0(X,\cF_i)$ and let $\cF_1$ and $\cF_2$ be generated by  global sections. Let $\varphi: \cF_1\otimes_{\cO_X} \cF_2 \to \cO_X$ be a pairing such that the induced bilinear mapping $V_1 \times V_2\to K$ is nondegenerate. Then:
 \begin{enumerate}[a)]
  \item $\cF_i\cong\cO_X^r$ where $r=\dim V_i$;
  \item the morphism $\cF_1\to\shHom_{\cO_X}(\cF_2,\cO_X)$ corresponding to $\varphi$ is an isomorphism.
 \end{enumerate}
\end{prop}

\begin{proof}
 Let $s_1,\ldots,s_r\in V_1$ be a basis of $V_1$. By assumption, $s_1,\ldots,s_r$ span the $\kappa(\xi)$-vector space $(\cF_1)_\xi$ and by Lemma \ref{lem:linindep} they are linearly independent. Thus by Lemma \ref{lem:basisfree} we have $\cF_1\cong\cO_X^r$. The same argument applies to $\cF_2$ and Part $b)$ then follows immediately from a).
\end{proof}

\begin{lem}\label{lem:injective1}
 Let $\cF_1$ and $\cF_2$ be coherent torsion-free sheaves on $X$. Assume that the image of $V_i=H^0(X,\cF_i)$ spans the stalk $(\cF_i)_\xi$ as $\kappa(\xi)$-vector space. Furthermore, let $\varphi: \cF_1\otimes_{\cO_X} \cF_2 \to \cO_X$ be a pairing such that the induced bilinear mapping $V_1 \times V_2\to K$ is nondegenerate. Then the corresponding morphism $\cF_1\to\shHom_{\cO_X}(\cF_2,\cO_X)$ is injective.
\end{lem}

\begin{proof}
 Let $s_1,\ldots,s_r\in V_1$ be a basis of $V_1$ and let $t_1,\ldots,t_r\in V_2$ be the dual basis with respect to the bilinear mapping $V_1 \times V_2\to K$. Let $U\subset X$ be some open affine subset, $A=\cO_X(U)$ and $M_i=\cF_i(U)$. For any $0\neq g\in M_1$ there is $0\neq a\in A$ such that $a\cdot g$ is in the submodule of $M_1$ that is spanned by the restrictions $s_i|_U$: $$a\cdot g=f_1\cdot s_1|_U+\ldots+f_r\cdot s_r|_U$$ for some $f_j\in A$ that are not all zero. Let for instance $f_i\neq0$, then $$\varphi(a\cdot g\otimes t_i|_U)=a\cdot f_i\cdot\varphi(s_i\otimes t_i)\neq0.$$ This shows that the map $M\to\Hom_A(M,A)$ induced by $\varphi$ in injective.
\end{proof}

\begin{lem}\label{lem:injective2}
 Let $\cF$ be a coherent sheaf on $X$ and $\cG$ a subsheaf with $\cG_\xi=\cF_\xi$. Then the natural map $\shHom_{\cO_X}(\cF,\cO_X)\to\shHom_{\cO_X}(\cG,\cO_X)$ is injective.
\end{lem}

\begin{proof}
 Let $U\subset X$ be some open affine subset, $A=\cO_X(U)$, $M=\cF(U)$ and $N=\cG(U)$. Consider a morphism $\varphi: M\to A$ such that $\varphi|_N=0$. For every $g\in M$ there is a nonzero $t\in A$ such that $t\cdot g\in N$. Thus $\varphi(t\cdot g)=t\cdot\varphi(g)=0$ and therefore $\varphi(g)=0$. This shows that $\varphi=0$.
\end{proof}

\begin{thm}\label{thm:freeifnondeg}
 Let $\cF_1$ and $\cF_2$ be coherent torsion-free sheaves on $X$ and let $V_i=H^0(X,\cF_i)$. Let $\varphi: \cF_1\otimes_{\cO_X}\cF_2\to\cO_X$ be a pairing such that the induced bilinear mapping $V_1 \times V_2\to K$ is nondegenerate. If $\dim V_1 \geq \rank \cF_1$, then $\cF_1\cong\cO_X^r$ and $r=\dim V_1 = \rank \cF_1$. Furthermore, $\cF_1\to\shHom_{\cO_X}(\cF_2,\cO_X)$ is an isomorphism.
\end{thm}

\begin{proof}
	By Proposition \ref{prop:ifglobgen}a) it suffices to show that $\cF_i$ is generated by global sections. Let $\cG_i$ be the subsheaf of $\cF_i$ generated by its global sections $V_i$. We get the {commutative} diagram  
 \[ \begin{tikzcd}
   \cG_1 \arrow{d} \arrow{r} & \shHom_{\cO_X}(\cG_2,\cO_X)  \\
   \cF_1 \arrow{r} & \shHom_{\cO_X}(\cF_2,\cO_X) \arrow{u}
  \end{tikzcd}. \]
 The homomorphism $\cG_1  \to \shHom_{\cO_X}(\cG_2,\cO_X)$ is an isomorphism by Proposition \ref{prop:ifglobgen}b). By Lemma \ref{lem:linindep} and the condition on the dimension it follows that the image of $V_i=H^0(X,\cF_i)$ spans the stalk $(\cF_i)_\xi$ as $\kappa(\xi)$-vector space and that $(\cG_i)_\xi=(\cF_i)_\xi$.
	Thus, the bottom and right maps in the diagram are also injective by Lemmas~\ref{lem:injective1} and \ref{lem:injective2}, respectively.
	This implies that $\cG_1=\cF_2$ and therefore $\cF_1$ is generated by global sections.
\end{proof}

\begin{ex}\label{ex:nondegcruc}
 This example is to illustrate that the assumption in \Cref{thm:freeifnondeg} of being nondegenerate on global sections is crucial.
 Let $X={\pp^1}$ and $\xi$ the generic point of ${\pp^1}$. Consider the coherent torsion-free sheaf $\cF=\cO_{\pp^1}(1)\oplus\cO_{\pp^1}(-1)$ on ${\pp^1}$. We have $\dim H^0({\pp^1},\cF) = 2 = \rank_\xi \cF$. On $\cF$ we define the pairing $\varphi: \cF\otimes_{\cO_X} \cF \to \cO_X$ that sends $(a,b)\otimes(c,d)$ to $ad+bc$ which is nondegenerate at $\xi$. But the induced bilinear form on the global sections of $\cF$ is identically zero and $\cF\not\cong\cO_{\pp^1}^2$.
\end{ex}

\section{Ulrich sheaves}\label{sec:ulr}

\begin{Def}
 Let $f:X\to Y$ be a morphism of schemes. A coherent sheaf $\cF$ on $X$ is \emph{$f$-Ulrich} if $f_*\cF\cong\cO_Y^r$ for some natural number $r$.
\end{Def}

\begin{rem}\label{rem:ulrichconnected}
 Let $f_i:X_i\to Y$ be finitely many morphisms of schemes. Let $X$ be the disjoint union of the $X_i$ and $f:X\to Y$ the morphism induced by the $f_i$. A coherent sheaf $\cF$ on $X$ is $f$-Ulrich if and only if $\cF|_{X_i}$ is $f_i$-Ulrich for all $i$. Thus we can usually restrict to the case when $X$ is connected.
\end{rem}

A case of particular interest is when $X$ is an embedded projective variety.

\begin{prop}[Prop.~2.1 in \cite{ES03}]\label{prop:ulrichchara}
 Let $X\subset\pp^n$ be a closed subscheme and $f: X\to\pp^{k}$ a finite surjective linear projection from a center that is disjoint from $X$. Let $\cF$ be a coherent sheaf whose support is all of $X$. Then the following are equivalent:
 \begin{enumerate}[(i)]
  \item $\cF$ is $f$-Ulrich;
  \item $H^i(X,\cF(-i))=0$ for $i>0$ and $H^i(X,\cF(-i-1))=0$ for $i<k$;
  \item $H^i(X,\cF(j))=0$ for all $1\leq i\leq k-1$, $j\in\Z$; $H^0(X,\cF(j))=0$ for $j<0$ and $H^k(X,\cF(j))=0$ for $j\geq -k$;
  \item the module $M=\oplus_{i\in\Z}H^0(X,\cF(i))$ of twisted global sections is a Cohen--Macaulay module over the polynomial ring $S=K[x_0,\ldots,x_n]$ of dimension $k+1$ whose minimal $S$-free resolution is linear.
 \end{enumerate}
\end{prop}

If $\cF$ as in \Cref{prop:ulrichchara} satisfies these equivalent conditions, then we say that $\cF$ is an \emph{Ulrich sheaf} on $X$ without specifying the finite surjective linear projection $f$ as conditions $(ii)$---$(iv)$ do not depend on the choice of $f$. A major open question in this context is the following:

\begin{prob}[p.~543 in \cite{ES03}]
 Is there an Ulrich sheaf on every closed subvariety $X\subset\pp^n$?
\end{prob}

We now want to apply the results from \Cref{sec:symm} to give a criterion for a sheaf to be Ulrich. For this we need a relative notion of nondegenerate bilinear mappings. Let $f: X \to Y$ be a finite morphism of noetherian schemes. For any quasi-coherent sheaf $\cG$ on $Y$ we consider the sheaf $\shHom_{\cO_Y}(f_* \cO_X, \cG)$. Since this is a quasi-coherent $f_* \cO_X$-module, it corresponds to a quasi-coherent $\cO_X$-module which we will denote by $f^! \cG$. We recall the following basic lemma (cf. \cite[III \S6, Ex.~6.10]{Hart77}).

\begin{lem}\label{lem:grothbasic}
 Let $f: X \to Y$ be a finite morphism of noetherian schemes.
 Let $\cF$ be a coherent sheaf on $X$ and $\cG$ be a quasi-coherent sheaf on $Y$. There is a natural isomorphism
 \[f_* \shHom_{\cO_X}(\cF, f^! \cG) \to \shHom_{\cO_Y}(f_* \cF, \cG)\]
 of quasi-coherent $f_*\cO_X$-modules.
\end{lem}

Let $f: X \to Y$ be a finite morphism of noetherian schemes.
Let $\cF_1$ and $\cF_2$ be coherent sheaves on $X$ and consider an $f^! \cO_Y$-valued pairing, i.e.,
a morphism  $\cF_1\otimes\cF_2 \to f^! \cO_Y$ of coherent $\cO_X$-modules.
This corresponds to a morphism $\cF_1 \to \shHom_{\cO_X}(\cF_2, f^! \cO_Y)$.
Lemma \ref{lem:grothbasic} tells us that this gives us a morphism 
\[
 f_* \cF_1 \to \shHom_{\cO_Y}(f_* \cF_2, \cO_Y)
\]
which corresponds to an $\cO_Y$-valued pairing on the pushforwards $f_* \cF_1$ and $f_*\cF_2$.

\begin{rem}
 Let us explain here the affine case in more detail. To that end let $X=\Spec(B)$, $Y=\Spec(A)$ and $f:X\to Y$ be induced by the finite ring homomorphism $f^\#:A\to B$. Then $f^!\cO_Y$ is the sheaf on $X$ associated to the $B$-module $\Hom_A(B,A)$ whose $B$-module structure is given by $(b\cdot\varphi)(m)=\varphi(b\cdot m)$ for all $b,m\in B$. If $\cF_1$ and $\cF_2$ are the coherent sheaves associated to the $B$-modules $M_1$ and $M_2$, then a morphism $\cF_1 \otimes\cF_2 \to f^! \cO_Y$ of coherent $\cO_X$-modules corresponds to a homomorphism $\psi:M_1\otimes_B M_2\to\Hom_A(B,A)$ of $B$-modules. This gives rise to the following $A$-bilinear map on $M_1\times M_2$: $$(m_1,m_2)\mapsto (\psi(m_1,m_2))(1).$$This gives the $\cO_Y$-valued pairing of the pushforwards $f_* \cF_1$ and $f_*\cF_2$.
\end{rem}

\begin{Def}
 Let $f: X \to Y$ be a finite morphism of noetherian schemes.
 Let $\cF_1$ and $\cF_2$ be coherent sheaves on $X$ and consider an $f^! \cO_Y$-valued pairing $\varphi:\cF_1\otimes\cF_2 \to f^! \cO_Y$.
 For a field $K$ we say that $\varphi$ is \textit{nondegenerate at $\alpha\in Y(K)$} if the induced $\cO_Y$-valued pairing of the pushforwards $f_* \cF_1$ and $f_*\cF_2$ is nondegenerate at $\alpha$.
\end{Def}

Now let $Y$ be a geometrically irreducible variety which is proper over a field $K$. Let further $f:X\to Y$ be a finite surjective morphism and $\cF_1$, $\cF_2$ coherent sheaves on $X$ with a pairing $\cF_1\otimes\cF_2\to f^! \cO_Y$. We have seen that this induces an $\cO_Y$-valued pairing on the pushforwards which in turn induces a $K$-bilinear mapping $$H^0(X,\cF_1)\times H^0(X,\cF_2)\to K.$$

\begin{thm}\label{thm:ulrichifnondeg}
 Let $X$ be an equidimensional variety over a field $K$ with  irreducible components $X_1,\ldots,X_s$.
 Let $\cF_1$ and $\cF_2$ be coherent torsion-free sheaves on $X$ and let $V_i=H^0(X,\cF_i)$.
 Let $Y$ be a geometrically irreducible variety which is proper over $K$ and $f:X\to Y$ a finite surjective morphism. Assume that there is an  $f^! \cO_Y$-valued pairing of $\cF_1$ and $\cF_2$ such that the induced $K$-bilinear mapping $V_1\times V_2\to K$ is nondegenerate. Then the following are equivalent:
 \begin{enumerate}[(i)]
  \item $\dim V_1 \geq\sum_{i=1}^s \deg(f|_{X_i}) \cdot \rank (\cF_1|_{X_i})$;
  \item $\cF_1$ is $f$-Ulrich.
 \end{enumerate}
\end{thm}

\begin{proof}
 We will apply \Cref{thm:freeifnondeg} to the coherent sheaves $f_*\cF_1$ and $f_*\cF_2$. First we need that the $f_*\cF_i$ are torsion-free. This follows from the assumptions that $f$ is finite and surjective, $X$ is equidimensional and $\cF_i$ is torsion-free. Further by \cite[VI \S2, Prop.~2.7]{kollar} one has $\rank (f_*\cF_1)=\sum_{i=1}^s \deg(f|_{X_i}) \cdot \rank (\cF_1|_{X_i}).$
\end{proof}

\section{The codifferent sheaf}\label{sec:codif}

In this section we recall some properties of $f^!\cO_Y$ and its relation to the codifferent sheaf. Most of the results should be well known, but for a lack of adequate references we will include (at least partial) proofs here.

\begin{lem}\label{lem:shriekinv}
 Let $f:X\to Y$ be a finite morphism of noetherian schemes.
 \begin{enumerate}[a)]
  \item If $f$ is flat and both $X$ and $Y$ are Gorenstein, then $f^!\cO_Y$ is a line bundle.
  \item If $Y$ is a smooth variety over $K$ and $X$ is Gorenstein, then $f^!\cO_Y$ is a line bundle.
 \end{enumerate}
\end{lem}

\begin{proof}
 We first note that $b)$ is a special case of part $a)$ because in this situation the morphism $f$ is automatically flat as Gorenstein implies Cohen-Macaulay. In order to prove part $a)$ note that for every $y \in Y$ the canonical module of $\cO_{Y,y}$ is $\cO_{Y,y}$ itself by \cite[Satz~5.9]{kanonisch}, and for every $x \in X$ the canonical module of $\cO_{X,x}$ is $(f^! \cO_Y)_x$ by \cite[Satz~5.12]{kanonisch}.
 Thus, again by \cite[Satz~5.9]{kanonisch}, $f^! \cO_Y$ is a line bundle on $X$ if $X$ is Gorenstein.
\end{proof}

\begin{rem}
 The preceding lemma implies that the sheaf $f^!\cO_Y$ is a line bundle whenever $f:X\to Y$ is a finite morphism of smooth varieties over a field $K$.
\end{rem}

\begin{Def}\label{def:codif}
 Let $A$ be a noetherian integral domain and $A\subset B$ be a finite ring extension such that for each minimal prime ideal $\fp$ of $B$ we have $\fp\cap A=(0)$. Let $K$ be the quotient field of $A$ and let $L$ be the total quotient ring of $B$. Then $L$ is a finite dimensional $K$-vector space and we have the $K$-linear map $\tr_{L/K}:L\to K$ that associates to every element $x\in L$ the trace of the $K$-linear map $L\to L,\, a\mapsto ax$.
 The \textit{codifferent} of the ring extension $A\subset B$ is the $B$-module
 $$\Delta(B/A)=\{g\in L:\, \tr_{L/K}(g\cdot B)\subset A\}.$$ Clearly, the map $$\Delta(B/A)\to\Hom_A(B,A),\, g\mapsto \tr_{L/K}(g\cdot -) $$ is a homomorphism of $B$-modules. 
 
	Now let $Y$ be an integral noetherian scheme and $f:X\to Y$ a finite morphism such the generic point of each irreducible component of $X$ is mapped to the generic point of $Y$. Let $\cK_X$ be the sheaf of total quotient rings of $\cO_X$.
 By glueing the above we define the quasi-coherent subsheaf $\Delta(X/Y)$ of $\cK_X$ and we obtain a morphism of $\cO_X$-modules $\Delta(X/Y)\to f^!\cO_Y$. We call $\Delta(X/Y)$ the \emph{codifferent sheaf}.
\end{Def}

\begin{ex}\label{ex:simplecodiff}
 Let $A$ be an integral domain and $K=\Quot(A)$. Let $f\in A[t]$ be a monic polynomial over $A$ which has only simple zeros in the algebraic closure of $K$ and let $B=A[t]/(f)$. Then the codifferent $\Delta(B/A)$ is the fractional ideal generated by $\frac{1}{f'}$ in the total quotient ring of $B$. Here $f'$ denotes the formal derivative of $f$. This follows from a lemma often attributed to Euler, see \cite[III, \S 6]{localfields}.
\end{ex}

\begin{rem}
 In order to construct a pairing $\cF_1\otimes\cF_2\to f^!\cO_Y$ it thus suffices by the discussion in \Cref{def:codif} to construct a pairing $\cF_1\otimes\cF_2\to\cK_X$ whose image is contained in $\Delta(X/Y)$.
\end{rem}

\begin{lem}\label{lem:codifshrieks}
 Let $X$ and $Y$ be varieties over a field of characteristic zero, $X$ equidimensional and $Y$ irreducible. Then for any finite surjective morphism $f:X\to Y$ the map $\Delta(X/Y)\to f^!\cO_Y$ is an isomorphism of $\cO_X$-modules.
\end{lem}

\begin{proof}
 We can reduce to the affine case: Let $A\subset B$ be a finite extension of finitely generated $K$-algebras such that $A$ is an integral domain and $\fp\cap A=(0)$ for all minimal prime ideals $\fp$ of $B$. Let $E$ and $F$ be the total quotient rings of $A$ and $B$, respectively. Since $B\otimes_A E$ is finite dimensional as an $E$-vector space and reduced, it is a ring of the form $F_1\times\cdots\times F_r$ for some finite field extensions of $F_i$ of $E$. Since no element of $A$ is a zero divisor in $B$, we actually have that $F=F_1\times\cdots\times F_r$. Thus we have an injective map $\Hom_A(B,A)\to\Hom_E(F,E)$ that is given by tensoring with $E$. By definition, the preimage of $\Hom_A(B,A)$ under the map $$\psi:F\to\Hom_E(F,E),\, a\mapsto\tr_{F/E}(a\cdot -)$$is exactly $\Delta(B/A)$. It thus suffices to show that $\psi$ is an isomorphism because then the restriction of $\psi$ to $\Delta(B/A)$ is the desired isomorphism $\Delta(B/A)\to\Hom_A(B,A)$. The map $\psi$ is the direct sum of the maps $$\psi_i:F_i\to\Hom_E(F_i,E),\, a\mapsto\tr_{F_i/E}(a\cdot -)$$ and thus it suffices to show that each $\psi_i$ is an isomorphism. We have that $\tr_{F_i/E}(1)$ is the dimension of $F_i$ as an $E$-vector space. Since we work over a field of characteristic zero, this shows that each $\psi_i$ is a nonzero map. But since $\Hom_E(F_i,E)$ is one dimensional considered as a vector space over $F_i$, this implies that $\psi_i$ is an isomorphism.
\end{proof}

\begin{rem}
 \Cref{lem:codifshrieks} is no longer true over fields of positive characteristic because the trace $\tr_{L/K}$ is identically zero for inseparable field extensions $K\subset L$.
 This is one reason why we have not worked with the codifferent sheaf to begin with.
\end{rem}

\begin{prop}\label{prop:ramdiv}
	Let $f:X\to Y$ be a finite surjective morphism of varieties over a field of characteristic zero. Let $X$ be equidimensional and Gorenstein and let $Y$ be smooth and irreducible. Then $\Delta(X/Y)$ is an invertible subsheaf of $\cK_X$ and thus $\Delta(X/Y)=\cL(R)$ for some Cartier divisor $R$ on $X$. This Cartier divisor is effective and its support consists exactly of those points where $f$ is ramified.
\end{prop}

\begin{proof}
 By \Cref{lem:codifshrieks} $\Delta(X/Y)$ is isomorphic to $f^!\cO_Y$ which is an invertible sheaf by \Cref{lem:shriekinv}. Thus $\Delta(X/Y)$ is an invertible subsheaf of $\cK_X$ and we can write $\Delta(X/Y)=\cL(R)$ for some Cartier divisor $R$ on $X$. We first show that $R$ is effective which is equivalent to the constant $1$ being a global section of $\Delta(X/Y)$. This can be checked locally.
 We thus consider a finite ring extension $A\subset B$ where $A$ is an integral domain. Furthermore, this ring extension is flat by the assumptions on $X$ and $Y$. Thus without loss of generality we can assume that $B$ is free as $A$-module. Therefore, the $A$-linear map $B\to B$ given by multiplication with an element $b\in B$ can be represented by a matrix having entries in $A$. Using the notation of \Cref{def:codif} this shows that $\tr_{L/K}(1\cdot B)\subset A$. Thus  the constant $1$ is a global section of $\Delta(X/Y)$ and $R$ is effective. The image of $1$ under the map $H^0(X,\Delta(X/Y))\to H^0(X,f^!\cO_Y)$ is just the trace map and the subscheme associated to $R$ is the zero set of this section. This consists exactly of the ramification points of $f$, see for example \cite[Tag~0BW9]{stacks-project} or \cite[Rem.~2.2.19]{KummerDiss}.
\end{proof}

\begin{Def}
 In the situation of \Cref{prop:ramdiv} we call the Cartier divisor $R$ on $X$ that corresponds to the invertible subsheaf $\Delta(X/Y)$ of $\cK_X$ the \emph{ramification divisor} of $f$.
\end{Def}

\begin{lem}\label{lem:codiffextension}
  Let $f:X\to Y$ be a finite surjective morphism of varieties. Let $X$ be equidimensional and let $Y$ be smooth and irreducible. Let $Z\subset X$ be of codimension at least two. Consider the open subset $V=X\setminus Z$ and its inclusion $\iota\colon V\to X$ to $X$. Then $\iota_*(\Delta|_V)=\Delta$ where $\Delta=\Delta(X/Y)$.
\end{lem}
\begin{proof}
If we enlarge $Z$, then the statement becomes stronger, so we may assume that $Z=\pi^{-1}(Z')$ for some $Z'\subset Y$ of codimension at least two. We write $\cD=\iota_*(\Delta|_V)$. Then $\Delta$ is a subsheaf of $\cD$ and we need to show equality. To that end let $U\subset Y$ be an affine open subset and $W=\pi^{-1}(U)$. Then we have the following:
\begin{enumerate}
  \item $\cO_X(W)\subset \cO_X(W\setminus Z)$ and
  \item $\cO_{Y}(U)=\cO_{Y}(U\setminus Z')$.
\end{enumerate}
Letting $L$ be the total quotient ring of $X$ and $K$ the function field of $Y$, we get
\[
\cD(W)=\Delta(\pi^{-1}(U\setminus Z'))=\{a\in L \mid \tr_{L/K}(a\cO_{X}(W\setminus Z))\subset \cO_{Y}(U\setminus Z')\,\}
\]
and due to (1) and (2) the latter is contained in
\[
\{\,a\in L \mid \tr_{L/K}(a\cO_{X}(W))\subset \cO_{Y}(U)\,\} = \Delta(W).
\]
Thus we have $\cD(W)=\Delta(W)$. Since $f\colon X\to Y$ is affine, the sets $W=\pi^{-1}(U)$ for $U\subset Y$ open and affine give an affine covering of $X$ and thus $\cD=\Delta$.
\end{proof}

\section{Real fibered morphisms}\label{sec:rf}
In this section we recall the notion of real fibered morphisms, basic examples and some of their properties.

\begin{Def}
 Let $f:X\to Y$ be a morphism of varieties over $\R$. If $f^{-1}(Y(\R))=X(\R)$, then we say that $f$ is real fibered.
\end{Def}

\begin{ex}\label{ex:nonneg}
 Let $p\in\R[x_0,\ldots,x_n]_{2d}$ be a homogeneous polynomial of degree $2d$. Inside the weighted projective space $\pp(d,1,\ldots,1)$ we consider the hypersurface $X$ defined by $y^2=p(x_0,\ldots,x_n)$ and the natural projection $\pi: X\to\pp^n$ onto the $x$-coordinates. This is a double cover of $\pp^n$ ramified at the hypersurface defined by $p=0$. Clearly $\pi$ is real fibered if and only if $p$ is globally nonnegative, i.e., $p(x)\geq0$ for all $x\in\R^{n+1}$.
\end{ex}

Hyperbolic polynomials yield another class of examples.

\begin{Def}
	Let $h\in\R[x_0,\ldots,x_n]_d$ be a homogeneous polynomial of degree $d$ and let $e\in\R^{n+1}$. We say that $h$ is \emph{hyperbolic with respect to $e$} if the univariate polynomial $h(te-v)\in\R[t]$ has only real zeros for all $v\in\R^{n+1}$. Note that this implies $h(e)\neq0$.
 A hypersurface $X\subset\pp^n$ is called \emph{hyperbolic with respect to $[e]$} if its defining polynomial is hyperbolic with respect to $e$.
\end{Def}

\begin{figure}[h]
  \begin{center}
    \includegraphics[width=4cm]{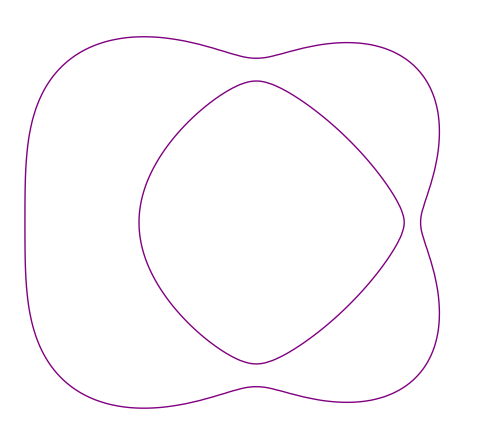}
  \end{center}
  \caption{A plane quartic curve that is hyperbolic with respect to any point in the inner oval.}
  \label{fig:hyp}
\end{figure}
  
\begin{ex}
 Let $X\subset\pp^n$ be a hypersurface and $e\in\pp^n$ a point that does not lie on $X$. Then the linear projection $\pi_e: X\to\pp^{n-1}$ with center $e$ is real fibered if and only if $X$ is hyperbolic with respect to $e$.
\end{ex}

One can generalize this notion naturally to varieties of higher codimension.

\begin{Def}
 Let $X\subset\pp^n$ be a variety of pure dimension $d$ and $E\subset\pp^n$ a linear subspace of codimension $d+1$ which does not intersect $X$. We say that $X$ is \emph{hyperbolic with respect to $E$} if the linear projection $\pi_E: X\to\pp^{d}$ with center $E$ is real fibered.
\end{Def}

An important feature of real fibered morphisms is the following.

\begin{thm}[Thm.~2.19 in \cite{eliich}]\label{thm:unramified}
 Let $f:X\to Y$ be a real fibered morphism of smooth varieties over $\R$. Then $f$ is unramified over $X(\R)$.
\end{thm}

The following property of real fibered morphisms will come in handy later when we want to construct positive semidefinite bilinear forms.

\begin{prop}\label{lem:tracelem}
 Let $Y$ be an irreducible smooth variety and let $X$ be an equidimensional variety over $\R$.
 Let $f:X\to Y$ be a finite surjective real fibered morphism. Let $K$ and $L$ be the total quotient rings of $Y$ and $X$ and $\tr_{L/K}:L\to K$ the trace map. If $g\in L$ is nonnegative on $X(\R)$ (whereever it is defined), then $\tr_{L/K}(g)$ is nonnegative on $Y(\R)$ (whereever it is defined).
\end{prop}

\begin{proof}
Assume that $g\in L$ is nonnegative on $X(\R)$ (whereever it is defined). By generic freeness \cite[Lem.~6.9.2]{EGAIV2} there is a nonempty open affine subset $U\subset Y$ such that $B=\cO_X(f^{-1}(U))$ is a free $A$-module where $A=\cO_Y(U)$ and such that $g\in B$. By a version of the Artin--Lang theorem \cite[Lem.~1.5]{Bec82} the real points of $U$ are dense in $Y(\R)$ with respect to the euclidean topology because $Y$ is smooth. Thus it suffices to show that $\tr_{L/K}(g)\in A$ is nonnegative on every real point $p$ of $U$. Let $\fm\subset A$ be the corresponding maximal ideal. Then $C=B/\fm B$ is finite dimensional as vector space over $\R=A/\fm$ and $\Spec(C)$ consists only of real points because $f$ is real fibered. Thus letting $\overline{g}\in C$ be the residue class of $g$ and  because $g$ is nonnegative on $f^{-1}(p)$, the quadratic form $$C\times C\to \R,\, (a,b)\mapsto\tr_{C/\R}(\overline{g}\cdot a\cdot b)$$ is positive semidefinite by \cite[Thm.~2.1]{pedersen}. Thus in particular $\tr_{C/\R}(\overline{g})\geq0$. Finally, by flatness, we have that $\overline{\tr_{L/K}(g)}=\tr_{C/\R}(\overline{g})\geq0$.
\end{proof}

\section{Symmetric and Hermitian bilinear forms}\label{sec:herm}

\begin{Def}
  Let $X$ be a scheme and let $\cF_1,\cF_2$ and $\cG$ be coherent sheaves on $X$. Let $\alpha:\cF_1\to\cF_2$ and $\beta:\cG\to\cG$ be isomorphisms of sheaves of abelian groups. 
  A $\cG$-{valued pairing} $\varphi: \cF_1 \otimes \cF_2 \to \cG$ is \emph{symmetric with respect to $\alpha$ and $\beta$} if we have for all local sections $s_i$ of $\cF_i$ that $$\varphi(s_1\otimes s_2)=\beta(\varphi(\alpha^{-1}(s_2)\otimes\alpha(s_1))).$$
\end{Def}

\begin{ex}
	Let $\cF=\cF_1=\cF_2$. Then $\varphi: \cF_1 \otimes \cF_2 \to \cG$ is \emph{(skew-)symmetric} if $\alpha:\cF\to\cF$ is the identity $\textrm{id}_\cF$ and $\beta:\cG\to\cG$ is $\textrm{id}_\cG$ ($-\textrm{id}_\cG$).
\end{ex}

In order to define hermitian bilinear forms  we have to set up some notation.
Let $A$ be an $\R$-algebra and $N'$ an $A$-module. Further consider $B=A\otimes_\R \C$ and the $B$-module $N=N'\otimes_A B=N'\otimes_\R \C$.
The complex conjugation on $\C$ induces an automorphism $\sigma_N: N\to N$ of $A$-modules whose fixed elements are exactly the elements of $N'$. 
If we have any $B$-module $M$, we can define another $B$-module $\overline{M}$. The elements of $\overline{M}$ are in bijection to those of $M$ and denoted by $\overline{x}$ for $x\in M$. The scalar multiplication is defined by letting $b\cdot \overline{x}:=\overline{\sigma_B(b)\cdot x}$. This implies that the map $\tau_M:M\to \overline{M},\,x\mapsto\overline{x}$ is an isomorphism of $A$-modules and we have $M=\overline{\overline{M}}$.

These definitions carry over to the case of schemes by glueing. If $X$ is a scheme over $\R$ we denote by $X_\C=X\times_{\Spec(\R)} \Spec(\C)$ the base change to $\C$ and $\pi:X_\C\to X$ the natural projection. If $\cG=\pi^*\cG'$ for some quasi-coherent sheaf $\cG'$ on $X$, we have an isomorphism $\sigma_\cG: \cG\to\overline{\cG}$ of sheaves of abelian groups as above.
For any quasi-coherent sheaf $\cF$ on $X_\C$ we obtain the quasi-coherent sheaf $\overline{\cF}$ together with an isomorphism of sheaves of abelian groups $\tau_\cF:\cF\to\overline{\cF}$. We say that a $\cG$-valued pairing $\cF\otimes\overline{\cF}\to\cG$ is \emph{hermitian} if it is symmetric with respect to $\tau_\cF$ and $\sigma_\cG$.

\begin{rem}
 Let $X$ be a scheme over $\R$. The complex conjugation induces an automorphism $\sigma:\cK_{X_\C}\to\cK_{X_\C}$ on the sheaf of total quotient rings of $X_\C$. If $\cF$ is a subsheaf of $\cK_{X_\C}$, then we can identify $\overline{\cF}$ with the image of $\cF$ under $\sigma$.
\end{rem}

\section{Positive semidefinite bilinear forms}\label{sec:psd}
The criterion for a coherent sheaf being Ulrich that we have seen in \Cref{sec:ulr} very much relies on the induced bilinear form on global sections being nondegenerate. Verifying this condition might be hard in general. However, in this section we show that when working over the real numbers we have a more convenient criterion at hand, namely positivity.
In this section $X$ will always be a variety over $\R$.

\begin{Def}\label{def:psd}
 Let $\cF$ be a coherent sheaf on $X$. Consider a symmetric bilinear form $\varphi: \cF\otimes_{\cO_X} \cF \to \cO_X$. Let $\alpha\in X(\R)$, i.e., a morphism $\alpha: \Spec(\R) \to X$. Then we get a symmetric bilinear form $\alpha^* \varphi$ on the pullback $\alpha^* \cF$ which is just a finite dimensional $\R$-vector space. We say that $\varphi$ is \textit{positive semidefinite at $\alpha\in X(\R)$} if $\alpha^* \varphi$ is positive semidefinite. We say that $\varphi$ is \textit{positive semidefinite} if it is positive semidefinite at every $\alpha\in X(\R)$.
 
 Analogously, let $\cF$ be a coherent sheaf on $X_\C$ and consider a hermitian bilinear form $\varphi: \cF\otimes_{\cO_X} \overline{\cF} \to \cO_X$. We can consider any $\alpha\in X(\R)$ also as a point of $X_\C$ that is fixed by the involution. Like this we obtain a hermitian bilinear form $\alpha^* \varphi$ on the pullback $\alpha^* \cF$ which is just a finite dimensional $\C$-vector space. We say that $\varphi$ is \textit{positive semidefinite at $\alpha\in X(\R)$} if $\alpha^* \varphi$ is positive semidefinite. We say that $\varphi$ is \textit{positive semidefinite} if it is positive semidefinite at every $\alpha\in X(\R)$.
\end{Def}

A symmetric or hermitian bilinear form on a coherent sheaf induces a symmetric or hermitian bilinear form on the space of global sections. The next lemma shows how this behaves with respect to positivity.

\begin{lem}\label{lem:globalpd}
 Let $X$ be irreducible and proper over $\R$ with generic point $\xi$.
 Let $\varphi: \cF\otimes_{\cO_X} \cF \to \cO_X$ be a positive semidefinite symmetric bilinear form on the coherent sheaf $\cF$ and $V=H^0(X,\cF)$.
 \begin{enumerate}[a)]
  \item  If $X(\R)\neq\emptyset$, then the induced bilinear form $V \times V\to \R$ positive semidefinite.
  \item If $X(\R)$ is Zariski dense in $X$, $\varphi$ is nondegenerate at $\xi$ and $\cF$ is torsion-free, then the induced bilinear form $V \times V\to \R$ is positive definite and thus in particular nondegenerate.
 \end{enumerate}
 The corresponding statements for hermitian bilinear forms hold true as well.
\end{lem}

\begin{proof}
 For part $a)$ we observe that if $\varphi(s,s)=-1$ for some $s\in V$, then $\varphi$ is not positive semidefinite at any point from $X(\R)$.
 
 In order to show part $b)$ we first observe that since $X(\R)$ is Zariski dense in $X$, the field $\kappa(\xi)$ has an ordering $P$. Consider a nonzero section $s\in V$. Since $\varphi(s,s)\geq0$ at all points in $X(\R)$, Tarski's principle implies that $\varphi(s,s)$ is also nonnegative with respect to the ordering $P$ when considered as an element of $\kappa(\xi)$. Thus the bilinear form induced by $\varphi$ on the $\kappa(\xi)$-vector subspace of $\cF_\xi$ spanned by $V$ is also positive semidefinite (with respect to $P$). But since $\varphi$ is nondegenerate at $\xi$, it is even positive definite (with respect to $P$). Finally, because $\cF$ is torsion-free, the nonzero section $s$ is mapped to a nonzero element in the stalk $\cF_\xi$ and therefore by positive definiteness $\varphi(s,s)\neq0$. This shows the claim.
\end{proof}

Thus if we assume positive semidefiniteness, we only need to assure that our bilinear form is nondegenerate at the generic point rather than on global sections.

\begin{thm}\label{thm:ulrichifpsd}
 Let $Y$ be a geometrically irreducible variety which is proper over $\R$. Let $f: X\to Y$ be a finite surjective morphism where $X$ is an equidimensional variety over $\R$ with $X(\R)$ Zariski dense in $X$. Let $X_1,\ldots, X_s$ be the irreducible components of $X$.
 Let $\cF$ be a coherent torsion-free sheaf on $X$ and let $V=H^0(X,\cF)$. Let $\varphi: \cF\otimes_{\cO_X} \cF \to f^!\cO_Y$ be a symmetric bilinear map which is nondegenerate at the generic point of each $X_i$. If the induced $\cO_Y$-valued bilinear form on $f_*\cF$ is positive semidefinite, then the following are equivalent:
  \begin{enumerate}[(i)]
   \item $\dim V \geq\sum_{i=1}^s \deg(f|_{X_i})\cdot \rank (\cF|_{X_i})$.
   \item $\cF$ is an $f$-Ulrich sheaf.
  \end{enumerate}
\end{thm}

\begin{proof}
 Combining \Cref{lem:globalpd} and \Cref{thm:freeifnondeg} it remains to show that the induced symmetric $\cO_Y$-valued bilinear form on $f_*\cF$ is nondegenerate at the generic point of $Y$. But this follows from the assumption that $\varphi$ is nondegenerate at the generic point of each $X_i$ by \Cref{lem:exc} whose easy proof we leave as an exercise.  
\end{proof}

\begin{lem}\label{lem:exc}
 Let $A=K_1\times\cdots\times K_r$ be the direct product of fields $K_i$ each of which is a finite extension of the field $K$. Let $M$ be a finitely generated $A$-module. Then $M\cong V_1\times\cdots\times V_r$ where each $V_i$ is a finite dimensional $K_i$-vector space and the right-hand side is considered as an $A$-module in the obvious way. A homomorphism $\varphi: M\to\Hom_K(M,K)$ of $A$-modules such that all induced maps $V_i\to\Hom_K(V_i,K)$ are isomorphisms is an isomorphism itself. In particular, the induced $K$-bilinear form $M\otimes_K M\to K$ is nondegenerate.
\end{lem}

In the same manner we obtain the hermitian version.

\begin{thm}\label{thm:ulrichifpsdherm}
 Let $Y$ be a geometrically irreducible variety which is proper over $\R$. Let $f: X\to Y$ be a finite surjective morphism where $X$ is an equidimensional variety over $\R$ with $X(\R)$ Zariski dense in $X$. Let $X_1,\ldots, X_s$ be the irreducible components of $X_\C$.
 Let $\cF$ be a coherent torsion-free sheaf on $X_\C$ and let $V=H^0(X,\cF)$. Let $\varphi: \cF\otimes_{\cO_X} \overline{\cF} \to f^!\cO_{Y_\C}$ be a hermitian bilinear map which is nondegenerate at the generic point of each $X_i$. If the induced $\cO_{Y_\C}$-valued hermitian bilinear form on $f_*\cF$ is positive semidefinite, then the following are equivalent:
  \begin{enumerate}[(i)]
   \item $\dim V \geq\sum_{i=1}^s \deg(f|_{X_i})\cdot \rank (\cF|_{X_i})$.
   \item $\cF$ is an $f$-Ulrich sheaf.
  \end{enumerate}
\end{thm}

\begin{Def}
	When the conditions, including those on $X,Y$ and $f$, of \Cref{thm:ulrichifpsd} or \ref{thm:ulrichifpsdherm} are satisfied we say that $\cF$ is a \emph{positive symmetric or hermitian} $f$-Ulrich sheaf, respectively.
\end{Def}

In order to check the positivity condition the following lemma will be useful.

\begin{lem}\label{thm:rank1}
 Let $Y$ be an irreducible smooth variety with $Y(\R)\neq\emptyset$ and let $X$ be an equidimensional variety over $\R$.
 Let $f: X\to Y$ be a finite surjective real fibered morphism.
 Let $s$ be a global section of $\cK_X$ which is nonnegative on $X(\R)$ and let $\cD$ be the subsheaf of $\cK_X$ given by $s\cdot\Delta(X/Y)$. Let $\cL$ be a subsheaf of $\cK_X$ such that $\cL\cdot\cL\subset\cD$. Then there is a $f^!\cO_Y$-valued symmetric bilinear form on $\cL$ such that the induced $\cO_Y$-valued bilinear form on $f_*\cL$ is positive semidefinite.
\end{lem}

\begin{proof}
 By assumption we can define the symmetric $\Delta(X/Y)$-valued bilinear form on $\cL$ that maps a pair of sections $(g,h)$ to $\frac{g\cdot h}{s}$. Thus the induced $\cO_Y$-valued bilinear form maps $(g,g)$ to the trace of $\frac{g^2}{s}$ which is nonnegative by \Cref{lem:tracelem}.
\end{proof}

\begin{rem}
 Let $\iota:X\hookrightarrow\pp^n$ be an embedding of a $k$-dimensional projective variety and let $\cF$ be an Ulrich sheaf on $X$, i.e. $\cF$ is $\pi$-Ulrich for some finite surjective linear projection $\pi:X\to\pp^k$. One can show that $$\iota_*(\shHom_{\cO_X}(\cF,\pi^!\cO_{\pp^k}))\cong\cE xt^{n-k}(\iota_*\cF,\omega_{\pp^n})(k+1).$$Thus our notion of symmetry coincides with the one introduced in \cite[\S 3.1]{ES03}. In particular, the property of being symmetric does not depend on the choice of the linear projection $\pi$ but the positivity condition does, as the next example shows.
\end{rem}

\begin{ex}
 Let $Y=\pp^1$ and $X=\cV(x_0^2-x_1^2-x_2^2)\subset\pp^2$. We let $f_i:X\to\pp^1$ be the linear projection with center $e_i$ where $e_1=[1:0:0]$ and $e_2=[0:1:0]$. Note that $f_1$ is real fibered but $f_2$ is not. Let $P=[1:1:0]\in X$ (Weil divisor on $X$) and $\cL=\cL(P)$ the corresponding invertible subsheaf of $\cK_X$. A basis of the space of global sections $V$ of $\cL(P)$ is given by the two rational functions $g_1=1$ and $g_2=\frac{x_0+x_1}{x_2}$. The ramification divisor $R_i$ of $f_i$ is given by $R_1=Q_0+\overline{Q_0}$ and $R_2=Q_1+Q_2$ with $Q_0=[0:1:\textrm{i}],\, Q_1=[1:0:1] \textrm{ and } Q_2=[1:0:-1].$ Denoting by $\Delta_i$ the codifferent sheaf associated to $f_i$, we thus get that \begin{align}\label{eqn:bil} (a,b)\mapsto a\cdot b\cdot \frac{x_0-x_1}{x_0}\end{align} defines a symmetric bilinear form $\cL\otimes\cL\to\Delta_1$. Via the isomorphism $\Delta_1\cong f_1^!\cO_{\pp^1}$ this induces an $\cO_{\pp^1}$-valued bilinear form on $(f_1)_*\cL$. With respect to the above basis of $V$ it is given by the matrix $$ \begin{pmatrix} 2 & 0\\0 &2                                                                                                                                                                                                                                                                                                                                                                                                                                                                                                                                                                                                                                                                                                                                                                                                                                                                                                                                                                                              \end{pmatrix}$$ which is positive definite. Thus $\cL$ is a positive symmetric $f_1$-Ulrich sheaf. On the other hand, via the isomorphism $\Delta_1\to\Delta_2$ that is given by multiplication with $\frac{x_0}{x_1}$ we get from ($\ref{eqn:bil}$) and $\Delta_2\cong f_2^!\cO_{\pp^1}$ an $\cO_{\pp^1}$-valued bilinear form on $(f_2)_*\cL$. With respect to the above basis of $V$ it is given by the matrix $$ \begin{pmatrix} -2 &0 \\ 0& 2                                                                                                                                                                                                                                                                                                                                                                                                                                                                                                                                                                                                                                                                                                                                                                                                                                                                                                                                                                                             \end{pmatrix}$$ which is not positive semidefinite. Thus although $f_1^!\cO_{\pp^1}\cong f_2^!\cO_{\pp^1}$ as abstract line bundles, for checking the positivity condition we need to specify the morphism. Moreover, \Cref{prop:posthenrf} will show that actually no nonzero symmetric $f_2^!\cO_{\pp^1}$-valued bilinear form on $\cL$ will induce a positive semidefinite bilinear form on $(f_2)_*\cL$ since $f_2$ is not real fibered.
\end{ex}

Since later on we will focus on irreducible varieties, we close this section with an example for the reducible case. A systematic study of positive Ulrich sheaves on reducible hypersurfaces would be very interesting with regard to the so-called \emph{generalized Lax conjecture}, see \cite[Con.~3.3]{victorsurvey}. We think it would be particularly beneficial to understand how the main result of \cite{kumbez} fits into our context here.

\begin{ex}\label{ex:quadsurface}
 Let $l=x_0-x_1$ and $h=x_0^2-(x_1^2+x_2^2+x_3^2)$.
 Let $X=X_1\cup X_2$ where $X_1=\cV(l)\subset\pp^3$ and $X_2=\cV(h)\subset\pp^3$. The linear projection $$f:X\to\pp^2,\,[x_0:x_1:x_2:x_3]\to[x_1:x_2:x_3]$$ with center $e=[1:0:0:0]$ is real fibered and we want to construct a positive symmetric $f$-Ulrich sheaf on $X$. The function field of $\pp^2$ is $K=\R(\frac{x_2}{x_1},\frac{x_3}{x_1})$ and the total quotient ring of $X$ is $L=L_1\times L_2$ where $L_i$ is the function field of $X_i$ for $i=1,2$. We note that $\cO_X$ can be identified with the following subsheaf of $\cK_X$: $$\cO_X(U)=\{(a,b)\in\cO_{X_1}(U\cap X_1)\times\cO_{X_2}(U\cap X_2):\, (a-b)|_{X_1\cap X_2\cap U}=0\}$$for $U\subset X$. We further define the subsheaf $\cP$ of $\cK_X$ via $$\cP(U)=\{(a,b)\in\cO_{X_1}(U\cap X_1)\times\cO_{X_2}(U\cap X_2):\, (a+b)|_{X_1\cap X_2\cap U}=0\}$$for $U\subset X$. Note that $\cP$ is a line bundle on $X$ and $\cP\cdot\cP=\cO_X$. Finally let \[\cE(U)=\{(a,b)\in\cO_{X_1}(U\cap X_1)\times\cO_{X_2}(U\cap X_2):\] \[ (a-\frac{x_3}{x_2}b)|_{X_1\cap X_2\cap U\cap U_2}=0 \textrm{ and }\]\[ (\frac{x_2}{x_3}a-b)|_{X_1\cap X_2\cap U\cap U_3}=0 \}.\]The coherent sheaf $\cE$ is defined in such a way that $\cE\cdot\cE\subset\cP$.
 
 Now let $V_i\subset\pp^2$ be the open affine set where $x_i\neq0$ and $U_i=f^{-1}(V_i)$ for $i=1,2,3$. By \Cref{ex:simplecodiff} the codifferent sheaf $\Delta(X/\pp^2)$ is the subsheaf of $\cK_X$ generated by $\frac{x_i^2}{\textrm{D}_e(l\cdot h)}$ on $U_i$ for $i=1,2,3$. As an element of $L_1\times L_2$ this is $$\left(-\frac{x_i^2}{x_2^2+x_3^2},\, \frac{x_i^2}{2\cdot(x_1^2+x_2^2+x_3^2-x_0x_1)}\right)$$ on $U_i$ for $i=1,2,3$. Let $\cL$ be the subsheaf of $\cK_X$ generated by $(\frac{x_i}{x_1},\frac{x_i}{x_1})$ on $U_i$. Consider $$s=\left(\frac{x_2^2+x_3^2}{x_1^2},\, \frac{2\cdot(x_1^2+x_2^2+x_3^2-x_0x_1)}{x_1^2}\right)$$ which is nonnegative on $X(\R)$. We have $\cL\cdot\cL\subset s\cdot\cP\cdot\Delta(X/\pp^2)$. Thus we have $\cF\cdot\cF\subset s\cdot \Delta(X/\pp^2)$ where $\cF=\cE\cdot\cL$. This gives us a $f^!\cO_{\pp^2}$-valued symmetric bilinear form on $\cL$ such that the induced $\cO_{\pp^2}$-valued bilinear form on $f_*\cF$ is positive semidefinite by \Cref{thm:rank1}. Here are linearly independent global sections of $\cF$: $$(0,x_0-x_1),\,(x_2,-x_3),\,(x_3,x_2).$$ Thus $\cF$ is a positive symmetric $f$-Ulrich sheaf.
\end{ex}

\section{Ulrich sheaves, determinantal representations and sums of squares}\label{sec:detreps}
The main reason why we are interested in Ulrich sheaves is because they correspond to certain determinantal representations. For the applications we consider in this article, we are interested in the following situation. Let $S=\R[y,x_1,\ldots,x_n]$ be the polynomial ring with the grading determined by letting $\deg(y)=e$ and $\deg(x_i)=1$ for $i=1,\ldots,n$, let $h\in S_{de}$ be an irreducible homogeneous element of degree $d\cdot e$ and $X=\cV(h)\subset\pp(e,1,\ldots,1)$ the hypersurface in the weighted projective space corresponding to $S$. Assume  $h(1,0,\cdots,0)=1$. The following proposition and the subsequent remark are well known among experts. But since we are not aware of a reference for the precise statement that we need, we include a proof for the sake of completeness. The case $e=1$ is treated for example in \cite[Thm.~A]{dethyp} and the proof for the general case works essentially in the same way.

\begin{prop}\label{prop:ulrichdetrep}
	Let $f:X\to\pp^{n-1}$ be the projection on the last coordinates and let $\cF$ be an $f$-Ulrich sheaf on $X$ with $\rank(\cF)=r$. Then there is a square matrix $A$ of size $d\cdot r$ 
	whose entries are homogeneous polynomials in the $x_i$ of degree $e$ such that $h^r=\det(y\cdot I-A)$. If $\cF$ is a positive symmetric (resp. hermitian) $f$-Ulrich sheaf, then $A$ can be chosen to be symmetric (resp. hermitian).
\end{prop}

\begin{proof}
	Let $\cO_X(1)$ be the pullback of $\cO_{\pp^{n-1}}(1)$ via $f$ and let $M=\oplus_{i\in\Z}H^0(X,\cF(i))$ be the module of twisted global sections. Since $\cF$ is an $f$-Ulrich sheaf, we have that $M$ considered as a module over $R=\R[x_1,\ldots,x_n]\subset S$ is isomorphic to $R^{d\cdot r}$.
	Multiplication with $y$ is an $R$-linear map, homogeneous of degree $e$, that can thus be represented by a square matrix $A$ of size $d\cdot r$ 
	whose entries are real homogeneous polynomials in the $x_i$ of degree $e$ and whose minimal polynomial is $h$. Thus $h^r=\det(y\cdot I-A)$. Further, a symmetric positive semidefinite bilinear form {as in Theorem~\ref{thm:ulrichifpsd}} yields a homomorphism of graded $S$-modules (of degree zero)
	$$\psi:M\to\Hom_R(M,R)$$
	which has the property that the induced $R$-bilinear form
	\begin{align*}
		R^{d\cdot r}\times R^{d\cdot r}&\to R\\
		(a,b)&\mapsto (\psi(a))(b)
	\end{align*}
	is symmetric, i.e., $(\psi(a))(b)=(\psi(b))(a)$. The degree zero part $V$ of $M\cong_R R^{d\cdot r}$ is the space of global sections of $\cF$ and the restriction of this symmetric bilinear form to $V$ is thus positive definite by \Cref{lem:globalpd}. We can therefore choose an orthonormal basis of $V$ with respect to this bilinear form. Note that this will also be a basis of the $R$-module $R^{d\cdot r}$ that is orthonormal with respect to $\psi$. Since $\psi$ is a homomorphism of $S$-modules, we have $(\psi(a))(y\cdot b)=(\psi(y\cdot a))(b)$, so multiplication with $y$ is self-adjoint with respect to our above defined symmetric bilinear form. Thus we can choose the representing matrix $A$ of the $R$-linear map given by multiplication with $y$ to be symmetric. The hermitian case follows analogously.
\end{proof}

\begin{rem}\label{rem:detrepthenulrich}
 The converse of \Cref{prop:ulrichdetrep} is also true: If $h^r=\det(y\cdot I-A)$, then the cokernel $M$ of the map $S^{d\cdot r}\to S^{d\cdot r}$ given by $y\cdot I-A$ is supported on $X$ and $M$ considered as $R$-module is just $R^{d\cdot r}$. If $A$ is symmetric, then an isomorphism of $M$ with $\Hom_R(M,R)$ as $S$-modules is given by sending the standard basis of $R^{d\cdot r}$ to its dual basis. We argue analogously in the hermitian case.
\end{rem}

A refined statement is true for possibly reducible subvarieties $X\subset\pp^n$ that are not necessarily hypersurfaces, see \cite[Thm.~5.7]{eliich}. These correspond to so-called determinantal representations of Livsic-type introduced in \cite{sha14} which are closely related to determinantal representations of the Chow form of $X$, see also \cite{ES03}. 
The main applications of this article however concern irreducible varieties only.

\begin{ex}
 The positive symmetric $f$-Ulrich sheaf on the reducible cubic hypersurface $X=\cV((x_0-x_1)\cdot(x_0^2-(x_1^2+x_2^2+x_3^2))\subset\pp^3$ from \Cref{ex:quadsurface} gives the following symmetric definite determinantal representation $$\begin{pmatrix}
                x_0+x_1& x_3 & -x_2\\                                                                                                                                                                                                                             
                x_3& x_0-x_1 & 0\\
-x_2& 0& x_0-x_1                \end{pmatrix}.$$
\end{ex}

\bigskip

Now we consider again the situation of \Cref{ex:nonneg} when $h$ is of the form $y^2-p(x)$ where $p\in\R[x_1,\ldots,x_n]_{2e}$ is a globally nonnegative polynomial.

\begin{lem}\label{lem:detrepthensos}
 Let $p\in\R[x_1,\ldots,x_n]_{2e}$ be a globally nonnegative polynomial which is not a square. Let $h=y^2-p(x)$ where $y$ has degree $e$.
 If $h^r=\det(y\cdot I-A)$ for some $r\geq 1$ and a symmetric or hermitian matrix $A$ of size $2r$ whose entries are homogeneous of degree $e$, then $p$ is a sum of $2r$ squares in the symmetric case and a sum of $4r-1$ squares in the hermitian case.
\end{lem}

\begin{proof}
 Since $p$ is not a square, we find that $h$ is irreducible. Thus $h^r=\det(y\cdot I-A)$ implies that $h$ is the minimal polynomial of $A$, i.e., $A^2=p\cdot I$. Letting $a_i$ be the $i$th column of $A$ we then have $p=a_i^ta_i$ in the symmetric case and $p=a_i^t\overline{a_i}$ in the hermitian case. Thus $p$ is a sum of $2r$ squares in the symmetric case and a sum of $4r-1$ squares in the hermitian case. 
\end{proof}

A converse of this last conclusion is given by the following lemma.

\begin{lem}\label{lem:cliff}
 Let $P=G_1^2+\ldots+G_n^2$ where the $G_i\in A$ are elements of some commutative ring $A$. There is a symmetric square matrix $Q$ of some size $m\in\N$ whose entries are $\Z$-linear combinations of the $G_i$ such that $P\cdot I_m=Q^2$.
\end{lem}

\begin{proof}
 This follows from the basic properties of Clifford algebras. Recall that the Clifford algebra $Cl_{0,n}(\R)$ is generated by $e_1,\ldots,e_n$ satisfying $e_i\cdot e_j=-e_j\cdot e_i$ if $i\neq j$ and $e_i^2=-1$. In particular, we have \[x_1^2+\ldots+x_n^2=-(x_1 e_1+\ldots+x_n e_n)\cdot(x_1 e_1+\ldots+x_n e_n)\] for all $x_k\in\R$. For $1\leq i\leq n$ let $A_i$ be the representing matrix of the map $$Cl_{0,n}(\R)\to Cl_{0,n}(\R),\, a\mapsto e_i\cdot a$$ with respect to the basis $e_{i_1}\cdots e_{i_r}$ of $Cl_{0,n}(\R)$ with $1\leq i_1<\cdots<i_r\leq n$ and $r\geq0$. Then one immediately verifies that $A_i$ is a matrix having only entries in $\{0,\pm1\}$ satisfying $A_i^T=-A_i$. It follows that \[(x_1^2+\ldots+x_n^2)\cdot I_N=S\cdot S^t\]for all $x_k\in\R$ where $S=x_1 A_1+\ldots+x_n A_n$ and $N$ is the dimension of $Cl_{0,n}(\R)$. Now we can choose $$Q=\begin{pmatrix}0&S\\S^t&0                                                                                                                                                                                                                                                                                                                                                                                                                                                                                                                                                                                                                                                                                                                                                                                                                                                                                                                       \end{pmatrix}
$$and we get $(x_1^2+\ldots+x_n^2)\cdot I_N=Q^2$ for all $x_k\in\R$.
  Since the entries of the $A_i$ are integers, the identity holds over every commutative ring.
\end{proof}

Putting all this together we get the following connection of sums of squares to Ulrich sheaves.

\begin{thm}\label{thm:ulrichsos}
 Let $p\in\R[x_1,\ldots,x_n]_{2e}$ be a homogeneous polynomial of degree $2e$ which is not a square. Inside the weighted projective space $\pp(e,1,\ldots,1)$ we consider the hypersurface $X$ defined by $y^2=p(x_0,\ldots,x_n)$ and the natural projection $\pi: X\to\pp^n$ onto the $x$-coordinates. Then $p$ is a sum of squares of polynomials if and only if there is a positive symmetric (or hermitian) $\pi$-Ulrich sheaf $\cF$ on $X$. In that case, if $\rank(\cF)=r$ then $p$ is a sum of $2r$ squares in the symmetric case and a sum of $4r-1$ squares in the hermitian case.
\end{thm}

\begin{proof}
 First assume that $p$ is a sum of squares. By \Cref{lem:cliff} there is a symmetric square matrix $A$ of some size $m\in\N$ whose entries are homogeneous of degree $e$ in the variables $x_1,\ldots,x_n$ such that $p\cdot I_m=A^2$. Because $p$ is not a square, the polynomial $h=y^2-p$ is irreducible. Thus $A^2-p\cdot I_m=0$ shows that $h$ is the minimal polynomial of $A$. This implies that $h^r=\det(y\cdot I-A)$ for some $r>0$. Thus there is a positive symmetric $\pi$-Ulrich sheaf on $X$ by \Cref{rem:detrepthenulrich}.
 
 Now assume that there is a positive symmetric (or hermitian) $\pi$-Ulrich sheaf $\cF$ on $X$. Then by \Cref{prop:ulrichdetrep} there is a symmetric (resp. hermitian) matrix $A$ of size $2\cdot r$ whose entries are homogeneous polynomials in the $x_i$ of degree $e$ such that $h^r=\det(y\cdot I-A)$. Now the claim follows from \Cref{lem:detrepthensos}.
\end{proof}

As a special case we get the following result by Netzer and Thom \cite{timundthom}.

\begin{cor}
 Let $h\in\R[x_0,\ldots,x_n]_2$ be a quadratic hyperbolic polynomial. Then $h^r$ has a definite symmetric determinantal representation for some $r>0$.
\end{cor}

\begin{proof}
 If $X=\cV(h)$ is hyperbolic with respect to $e\in\pp^n$, then the linear projection $\pi_e: X\to\pp^{n-1}$ is a real fibered double cover ramified along the zero set of a nonnegative quadratic polynomial $p$. Since $p$ is a sum of squares, there is a positive symmetric $\pi_e$-Ulrich sheaf on $X$ by \Cref{thm:ulrichsos}. \Cref{prop:ulrichdetrep} implies the claim.
\end{proof}

\begin{rem}
 Positive Ulrich sheaves on \emph{reciprocal linear spaces}, i.e. the closure of the image of a linear space under coordinatewise inversion, were used in \cite{bichow} to prove that a certain polynomial associated to a hyperplane arrangement, called the \emph{entropic discriminant}, is a sum of squares. The relation of Ulrich sheaves and sums of squares used in \cite{bichow} is a generalization of one direction of \Cref{thm:ulrichsos}. Namely let $f:X\to\pp^n$ be a finite surjective real fibered morphism such that $f_*\cO_X$ is a sum of line bundles. Then, if there is a positive $f$-Ulrich sheaf, then the polynomial defining the branch locus of $f$ is a sum of squares \cite[Thm.~6.1]{bichow}.
\end{rem}

\section{Positive Ulrich sheaves of rank one on irreducible varieties}\label{sec:smooth}
In this section let $f: X\to Y$ always denote a finite surjective morphism of geometrically irreducible varieties which are proper over $\R$. We assume that $Y$ is smooth and has a real point. We further assume that the singular locus of $X$ has codimension at least two.
This allows us to speak about Weil divisors on $X$. For a given Weil divisor $D$ on $X$ we denote by $\cL(D)$ the subsheaf of $\cK_X$ consisting of all rational functions with pole and zero orders prescribed by $D$. We further denote $\ell(D)=\dim(\Gamma(X,\cL(D)))$. Let $Z\subset Y$ be a closed subset of codimension at least two such that $f^{-1}(Z)$ contains $X_{\textnormal{sing}}$. Letting $V=Y\setminus Z$ and $U=f^{-1}(V)$, the restriction $f|_U: U\to V$ is a finite surjective morphism of smooth irreducible varieties. Thus $\Delta(U/V)$ is invertible by \Cref{prop:ramdiv} and thus corresponds to a Weil divisor $R$ on $U$. Since $X\setminus U$ has codimension at least two, we can also consider $R$ as a Weil divisor on $X$ which we call the ramification divisor of $f$. \Cref{lem:codiffextension} shows that the associated subsheaf $\cL(R)$ of $\cK_X$ is exactly $\Delta(X/Y)$. All this  holds true for the complexification $X_\C$ as well. Complex conjugation gives an involution $\sigma:X_\C\to X_\C$ and for a Weil divisor $D$ on $X_\C$ we denote $\sigma(D)$ by $\overline{D}$. Using the notation of \Cref{sec:herm} we have $\cL(\overline{D})=\overline{\cL(D)}$ as subsheaves of $\cK_{X_\C}$.

\begin{rem}
 If $D_1$ and $D_2$ are Weil divisors on $X$, then we have $$\cL(D_1)\cdot\cL(D_2)\subset\cL(D_1+D_2)$$considered as subsheaves of $\cK_X$. But since $\cL(D_i)$ is not necessarily an invertible sheaf, we do not have equality in general.
\end{rem}

The relative notion of positive semidefiniteness introduced in \Cref{sec:psd} relates to the notion of being real fibered in the following way.

\begin{prop}[Thm.~5.11 in \cite{eliich}]\label{prop:posthenrf}
 The following are equivalent:
 \begin{enumerate}[(i)]
  \item  There is a coherent sheaf $\cF$ on $X$ with $\Supp(\cF)=X$ and a symmetric nonzero $f^!\cO_Y$-valued bilinear form on $\cF$ such that the induced bilinear form on $f_*\cF$ is positive semidefinite,
  \item $f$ is real fibered, i.e., $f^{-1}(Y(\R))=X(\R)$.
 \end{enumerate}
\end{prop}

In this situation we get the following convenient criterion for a Weil divisor to give rise to a positive Ulrich sheaf.

\begin{thm}\label{thm:weil}
 Let $f: X\to Y$ be a real fibered finite surjective morphism of geometrically irreducible varieties which are proper over $\R$. Let $Y$ be smooth and have a real point. Further let the singular locus of $X$ have codimension at least two.
 Let $R$ be the ramification divisor of $f$ and let $s$ be a rational function on $X$ which is nonnegative on $X(\R)$.
 \begin{enumerate}
  \item If $D$ is a Weil divisor on $X$ such that $2D+(s)=R$ and $\ell(D)\geq\deg(f)$, then the sheaf $\cL(D)$ on $X$ is a positive symmetric $f$-Ulrich sheaf.
  \item If $D$ is a Weil divisor on $X_\C$ that satisfies $D+\overline{D}+(s)=R$ and $\ell(D)\geq\deg(f)$, then the sheaf $\cL(D)$ on $X_\C$ is a positive hermitian $f$-Ulrich sheaf.
 \end{enumerate}
\end{thm}

\begin{proof}
 Let $\cL=\cL(D)$ be the subsheaf of $\cK_X$ corresponding to $D$. By assumption we can define the symmetric resp. hermitian $\Delta(X/Y)$-valued bilinear form on $\cL$ that maps a pair of sections $(g,h)$ to the product $s\cdot g\cdot h$. This is clearly nondegenerate at the generic point of $X$ and it is positive semidefinite by \Cref{lem:tracelem}. Then the claim follows from \Cref{thm:ulrichifpsd} and \Cref{thm:ulrichifpsdherm} respectively.
\end{proof}

\begin{rem}\label{rem:signonclass}
 Let $D$ be a Weil divisor on $X$ or $X_\C$ such that $2D$ or $D+\overline{D}$, respectively, is linearly equivalent to $R$, i.e., they differ only by a principal divisor $(g)$. The signs that $g$ takes on real points of $X$ (up to global scaling) do only depend on the divisor class of $D$. Indeed, if $D'=D+(f)$ for some rational function $f$, then $2D'$ resp. $D+\overline{D}$ differs from $R$ by $g\cdot f^2$ or $g\cdot f\overline{f}$ respectively.
\end{rem}

\begin{ex}
 Let $L\subset\pp^n$ be a linear subspace of dimension $d<n$ that is not contained in any coordinate hyperplane. We denote by $L^{-1}$ its \emph{reciprocal}, i.e. the (Zariski closure of the) image of $L$ under the rational map $\pp^n\ratto\pp^n$ defined by coordinatewise inversion. It was shown by Varchenko \cite{Var95} that $L^{-1}$ is hyperbolic with respect to the orthogonal complement $L^\perp$ of $L$. Further it was shown in \cite{bichow} that there is a symmetric positive $f$-Ulrich sheaf of rank one on $L^{-1}$ where $f:X\to\pp^d$ is the linear projection from $L^\perp$. We want to outline how this also follows from \Cref{thm:weil}, at least for generic $L$. To this end let $L$ be the row span of a matrix $A=(a_{ij})$ of size $(d+1)\times(n+1)$ and assume that every maximal minor of $A$ is nonzero. Letting $l_j=\sum_{i=1}^{d+1}a_{ij}x_i$ for $j=1,\ldots,d+1$ we can describe $L^{-1}$ as the image of the rational map $$\psi:\pp^d\ratto\pp^n,\, x\mapsto(\frac{l_1\cdots l_{d+1}}{l_1}:\cdots:\frac{l_1\cdots l_{d+1}}{l_{d+1}}).$$ Note that $\psi$ is defined in all points where at most one of the $l_j$ vanishes. It follows from the proof of \cite[Cor.~5]{EntroDisc} and \cite[Rem.~31]{EntroDisc} that the ramification divisor $R$ of $f$ on $L^{-1}$ is the proper transform under $\psi$ of the zero set $Z\subset\pp^d$ of $$P=\sum_I \det(A_I)^2\prod_{j\in I}l_j^2$$ where the sum is taken over all $I\subset\{1,\ldots,n+1\}$ of size $n-d$ and $A_I$ denotes the submatrix of $A$ obtained from erasing all columns indexed by $I$. Let $H\subset\pp^d$ be a generic hyperplane defined by a linear form $G$ and let $D$ be the divisor on $L^{-1}$ defined as the proper transform of $H$ under $\psi$. On $\pp^d$ we have $2(n-d)H = Z + \left(\frac{P}{G^{2(n-d)}}\right)$ and our genericity assumption on $A$ implies that $P$ does not vanish entirely on any of the $\cV(l_i,l_j)\subset\pp^d$ which comprise the locus where $\psi$ is not regular. Therefore, we have $2(n-d)D=R+\left(\frac{P}{G^{2(n-d)}}\right)$ as divisors on $L^{-1}$. Clearly, the rational function $\frac{P}{G^{2(n-d)}}$ is nonnegative and $\ell((n-d)\cdot D)$ equals $\binom{n}{d}$, the number of monomials of degree $n-d$ in $d+1$ variables. Since this is also the degree of $L^{-1}$ \cite{PS06}, \Cref{thm:weil}(1) implies that $\cL((n-d)\cdot D)$ is a positive symmetric $f$-Ulrich sheaf on $L^{-1}$.
\end{ex}

The previous example leads to the following question. Let $X\subset\pp^n$ be an irreducible variety not contained in any coordinate hyperplane which is hyperbolic with respect to every linear subspace of codimension $\dim(X)+1$ all of whose Pl\"ucker coordinates are positive. Then the image $X^{-1}$ of $X$ under coordinatewise inversion is hyperbolic with respect to all these subspaces as well \cite[Prop.~1.4]{bichow}.

\begin{prob}\label{prob:reci}
 Given a symmetric positive Ulrich bundle on a variety $X\subset\pp^n$ as above. Does there exist one on $X^{-1}$ of the same rank as well?
\end{prob}

\begin{rem}
Using \cite[Prop.~3.3.8]{master} one can show that the answer to \Cref{prob:reci} is \emph{yes} for hypersurfaces. It is also true for $X$ a linear subspace by \cite{bichow}.
\end{rem}

In the case of hypersurfaces, \Cref{thm:weil} has a geometric interpretation in terms of so-called interlacers.

\begin{Def}
 Let $g,h\in\R[x_0,\ldots,x_n]$ with $d=\deg(g)+1=\deg(h)$ be hyperbolic with respect to $e$. If for all $v\in\R^{n+1}$ we have that $$a_1\leq b_1\leq a_2\leq\cdots\leq b_{d-1}\leq a_d$$ where the $a_i$ and $b_i$ are the zeros of $h(te-v)$ and $g(te-v)$ respectively, we say that $g$ \emph{interlaces} $h$, or that $g$ is an \emph{interlacer} of $h$. This definition carries over to hyperbolic hypersurfaces in the obvious way.
\end{Def}

\begin{figure}[h]
  \begin{center}
    \includegraphics[width=5cm]{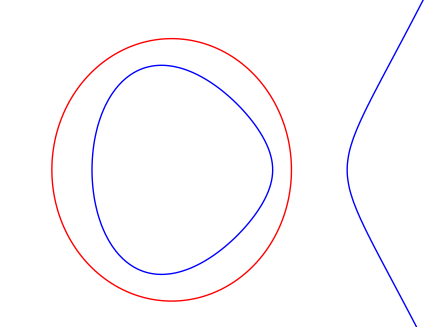}
  \end{center}
  \caption{A cubic hyperbolic plane curve (blue) interlaced by a plane hyperbolic conic (red).}
  \label{fig:interl}
\end{figure}

\begin{ex}
 If $h\in\R[x_0,\ldots,x_n]$ is hyperbolic with respect to $e$, then the directional derivative $\textrm{D}_eh$ of $h$ in direction $e$ is an interlacer of $h$. This follows from Rolle's theorem.
\end{ex}

\begin{cor}\label{cor:dixon}
 Let $h\in\R[x_0,\ldots,x_n]_d$ be hyperbolic with respect to $e$ and let $X=\cV(h)\subset\pp^n$ be the corresponding hypersurface. Assume that the singular locus of $X$ has dimension at most $n-3$. Let $g$ be an interlacer of $h$ and denote by $G$ the Weil divisor it defines on $X$.
 \begin{enumerate}
  \item Assume that $G=2D$ for some Weil divisor $D$ on $X$. If the vector space of all $p\in\R[x_0,\ldots,x_n]_{d-1}$ that vanish on $D$ has dimension at least $d$, then $h$ has a definite symmetric determinantal representation.
  \item Assume that $G=D+\overline{D}$ for some Weil divisor $D$ on $X_\C$. If the vector space of all $p\in\C[x_0,\ldots,x_n]_{d-1}$ that vanish on $D$ has dimension at least $d$, then $h$ has a definite hermitian determinantal representation.
 \end{enumerate}
\end{cor}

\begin{proof}
 The ramification divisor of the linear projection with center $e$ is the zero set of the directional derivative $\textrm{D}_eh$ on $X$. Since $g$ is an interlacer, we have that the rational function $\frac{\textrm{D}_eh}{g}$ is nonnegative on $X(\R)$ by \cite[Lem.~2.4]{interl}. Then the claim follows from \Cref{thm:weil} and \Cref{prop:ulrichdetrep} by our dimensional assumption. 
\end{proof}

\begin{rem}
 When $n=2$, the case of plane curves, the dimensional condition in \Cref{cor:dixon} is automatically satisfied by Riemann--Roch. If $g=\textrm{D}_eh$, then $G=D+\overline{D}$ by \Cref{thm:unramified}. In that case this recovers the result from \cite[\S 4]{hvelemt}.
 Note that the proof in \cite{hvelemt} does not seem to generalize to higher dimensions as it uses Max Noether's AF+BG theorem. 
 It would be interesting to know when we actually need the dimensional condition in the case of higher dimensional hypersurfaces.
\end{rem}

\begin{prob}
  Find hyperbolic hypersurfaces $X\subset\pp^n$ for which there is an interlacer whose zero divisor on $X$ is of the form $D+\overline{D}$ with $\ell(D)<\deg(X)$.
\end{prob}

\begin{ex}
 Let $h\in\R[x_0,\ldots,x_n]_n$ be the elementary symmetric polynomial of degree $n$ in $x_0,\ldots,x_n$. It is hyperbolic with respect to every point in the positive orthant and $g=\frac{\partial}{\partial x_0}h$ is an interlacer. The zero divisor of $g$ on the hypersurface $X=\cV(h)\subset\pp^n$ is of the form $2D$ where $D=\sum_{1\leq i<j\leq n} L_{ij}$. Here $L_{ij}$ is the linear subspace $\cV(x_i,x_j)\subset X$. Each monomial $\frac{x_1\cdots x_n}{x_i}$ for $1\leq i\leq n$ vanishes on $D$. Thus by \Cref{cor:dixon} the polynomial $h$ has a definite symmetric determinantal representation. This was known before and follows e.g. from the more general result that the bases generating polynomial of a regular matroid (in our case $U_{n,n+1}$) has a definite symmetric determinantal representation, see  \cite[\S 8.2]{COS04}.
\end{ex}

\section{Smooth curves}\label{sec:curves}
In this section let $f: X\to Y$ be a real fibered morphism between smooth irreducible curves that are projective over $\R$ with $Y(\R)$ Zariski dense in $Y$. Let $R$ be the ramification divisor of $f$. We want to apply \Cref{thm:weil}.

\begin{lem}\label{lem:narrowsquare}
 There is a divisor $M$ on $X$ and a nonnegative $s$ in the function field of $X$ such that $R+(s)=2M$.
\end{lem}

\begin{proof}
 The proof is a projective version of the proof of \cite[Cor.~4.2]{christoph}. By \Cref{thm:unramified} we have that $f$ is unramified at real points. Thus $R$, considered as a Weil divisor, is a sum of nonreal points. Therefore, the Weil divisor $R'$ that we obtain on the complexification $X_\C=X\times_\R \C$ of $X$ is of the form $R'=Q+\overline{Q}$ where $Q$ is some effective divisor and $\overline{Q}$ its complex conjugate. Since the group $\textnormal{Pic}^0(X_\C)$ is divisible, there is a $g$ in the function field of $X_\C$ and a divisor $N$ on $X_\C$ such that $Q-nP=2N+(g)$ where $n=\deg(Q)$ and $P$ is any point on $X_\C$ with $P=\overline{P}$. Thus $(Q-nP)+(\overline{Q-nP})=2(N+\overline{N})+(g\cdot\overline{g})$ which implies $$R'=Q+\overline{Q}=2(N+\overline{N}+nP)+(g\cdot\overline{g}).$$Since $N+\overline{N}+nP$ is fixed by conjugation, it descends to a divisor $M$ on $X$. The function $s=g\cdot\overline{g}$ is a sum of two squares and thus nonnegative.
\end{proof}

From this we get the following theorem.

 \begin{thm}\label{thm:curvetop1}
 Let $X$ be a smooth irreducible curve that is projective over $\R$. For every real fibered $f:X\to\pp^1$ there is a positive symmetric $f$-Ulrich line bundle.
\end{thm}

\begin{proof}
 Let  $M$ be the divisor from \Cref{lem:narrowsquare}. By \Cref{thm:weil} we have to show that $\ell(M)\geq\deg(f)$. By Hurwitz's Theorem we have that $2g-2=\deg(R)-2\deg(f)$ where $g$ is the genus of $X$. Thus $\deg(M)=\deg(f)+g-1$ and by Riemann--Roch \[\ell(M)\geq \deg(M)-g+1=\deg(f).\qedhere\]
\end{proof}

\begin{cor}[Thm.~7.2 in \cite{sha14}]
 The Chow form of every hyperbolic curve $X\subset\pp^n$ has a symmetric and definite determinantal representation.
\end{cor}

\begin{proof}
 By \cite[Thm.~5.7]{eliich} and \cite[Rem.~4.4]{eliich} it suffices to show that there is a positive Ulrich bundle of rank $1$ on $X$. But this follows from the preceding theorem applied to the linear projection from an $n-2$-space of hyperbolicity if $X$ is smooth. Otherwise we can pass to the normalization of $X$.
\end{proof}

\begin{cor}[Helton--Vinnikov Theorem \cite{HV07}]
 Every hyperbolic polynomial in three variables has a definite determinantal representation.
\end{cor}

\begin{ex}
 If the target is not $\pp^1$ as in \Cref{thm:curvetop1}, then there is in general no (positive symmetric) $f$-Ulrich sheaf on $X$. This fails already in the next easiest case, namely when $X$ and $Y$ both have genus one. Indeed, let $f:X\to Y$ be an unramified and real fibered double cover of an elliptic curve $Y$. Such maps exists, see for example \cite[Lem.~6.5]{kummer2020hyperbolic}, and by Riemann--Hurwitz $X$ is an elliptic curve as well. We claim that in this case there is actually no $f$-Ulrich sheaf at all. Indeed, there is a line bundle $\cL$ on $Y$ which is nontrivial and $2$-torsion such that $f_*\cO_X=\cO_Y\oplus\cL$. Then we have $f^*\cL=\cO_X$ and the projection formula implies that $f_*\cF=\cL\otimes f_*\cF$ for all coherent sheaves $\cF$ on $X$. This excludes $f_*\cF=\cO_Y^r$.
\end{ex}

\section{Del Pezzo Surfaces}\label{sec:delpezzo}
Recall that a \emph{del Pezzo surface} is a smooth projective and geometrically irreducible surface whose anticanonical class is ample. We are interested in morphisms $f: X\to\pp^2$ where $X$ is a del Pezzo surface and the pullback $f^*\cO_{\pp^2}(1)$ is the anticanonical line bundle. It was shown in \cite[Prop.~4.1]{ulrichintro} that for such $f$ there exist $f$-Ulrich line bundles. We will show that if $f$ is real fibered, then there are even positive hermitian $f$-Ulrich line bundles. An introduction to the classical theory of del Pezzo surfaces can be found for example in \cite[Ch.~8]{cag} or \cite[\S3.5]{nearly}. This section further relies on the classification of real del Pezzo surfaces \cite{russo}.

\begin{Def}
 The \emph{degree} of a del Pezzo surface $X$ is the self-intersection number $K_X.K_X$ of its canonical class $K_X$. A \emph{line} on $X$ is an irreducible curve $L\subset X$ such that $L.L=L.K_X=-1$.
\end{Def}

\begin{rem}
 Note that if the anticanonical class $-K_X$ of a del Pezzo surface $X$ is very ample, then a line $L$ on $X$ is mapped by the associated embedding to an actual line, i.e. a linear subspace of dimension one because $L.K_X=-1$.
\end{rem}

\begin{ex}\label{ex:pezzo}
 These are examples of del Pezzo surfaces \cite[Thm.~3.36(7)]{nearly}:
 \begin{enumerate}
  \item A smooth hypersurface of degree four in the weighted projective space $\pp(2,1,1,1)$ is a del Pezzo surface of degree two.
  \item A smooth cubic hypersurface in $\pp^3$ is a del Pezzo surface of degree three.
  \item A smooth complete intersection of two quadrics in $\pp^4$ is a del Pezzo surface of degree four.
 \end{enumerate}
 Furthermore, in case (1) the (complete) anticanonical linear system corresponds to the projection $\pp(2,1,1,1)\dashrightarrow\pp^2$ restricted to our surface. Moreover, in the cases (2) and (3) the embeddings of the surfaces to $\pp^3$ and $\pp^4$ respectively correspond to the (complete) anticanonical linear system. These statements are for example shown in the course of the proof of \cite[Thm.~3.36]{nearly}.
\end{ex}

\begin{rem}\label{rem:blowup}
 If $X$ is a del Pezzo surface over an algebraically closed field, then $X$ is isomorphic to either $\pp^1\times\pp^1$ or a blowup of $\pp^2$ in $n\leq8$ points, see \cite[Ex.~3.56]{nearly} or \cite[Cor.~8.1.17, Prop.~8.1.25]{cag}. A straight-forward calculation shows that the degree of $X$ is eight in the former and $9-n$ in the latter case.
\end{rem}

\begin{lem}\label{lem:quarticsurface}
 Let $X\subset\pp^4$ be a smooth complete intersection of two quadrics in $\pp^4$ such that $X(\R)$ is homeomorphic to the disjoint union of two spheres. Then:
 \begin{enumerate}
  \item $X$ is contained in exactly five real singular quadrics;
  \item one of these five quadrics has signature $(2,2)$ and the other four have signature $(3,1)$;
  \item for exactly two of these singular quadrics the linear projection from its vertex realizes $X$ as a real fibered double cover of a hyperbolic quadratic hypersurface $Q\subset\pp^3$.
 \end{enumerate}
\end{lem}

\begin{proof}
 The complex pencil $\lambda q_0+\mu q_1$ contains five singular quadrics as the $q_i$ can be represented by symmetric $5\times5$ matrices. We will show that all of them are real. To this end recall that by \Cref{rem:blowup} the complexification $X_\C$ of $X$ is isomorphic to the blowup of $\pp^2$ at five points $p_0,\ldots,p_4$. The 16 lines on $X_\C$ correspond to the exceptional divisors $E_0,\ldots,E_4$, the lines $l_{ij}$ through $p_i$ and $p_j$ for $0\leq i<j\leq4$ and the conic $C$ through all five points $p_0,\ldots,p_4$, see for example \cite[Thm.~26.2]{manin}. After relabeling if necessary, the complex conjugation on $X_\C$ interchanges $E_0$ with $C$, $E_i$ with $l_{0i}$ for $1\leq i\leq4$ and $l_{ij}$ with $l_{kl}$ for $\{i,j,k,l\}=\{1,2,3,4\}$ and $i<j$, $k<l$, see \cite[Exp.~2, case $n=3$]{russo}. We write $$A_1=E_0+C, A_2=l_{12}+l_{34}, A_3=l_{13}+l_{24}, A_4=l_{14}+l_{23}$$ (divisors on $X$) and note that all $A_i$ belong to the same linear system. Similarly, the divisors $B_i=E_i+l_{0i}$ for $1\leq i\leq4$ on $X$ are also linearly equivalent to each other. Note that we did write $A_i=L_i+\overline{L_i}$ and $B_i=L_i'+\overline{L_i'}$ for suitable lines $L_i$ and $L_i'$ on $X_\C$. Each of the two linear systems realize $X$ as a conic bundle, i.e., define a morphism $X\to\pp^1$ all of whose fibers are isomorphic to a plane conic curve \cite[Exp.~2]{russo}. The four singular fibers of each bundle are exactly the $A_i$ and $B_i$ respectively. Therefore, for each connected component $S$ of $X(\R)$ there are exactly two values for $i$ and $j$ such that $L_i\cap\overline{L_i}$ resp. $L_j'\cap \overline{L_j'}$ is a point on $S$. Our two conic bundle structures on $X$ induce a map $X\to\pp^1\times\pp^1$ which is a double cover since $A_i\cdot B_j=2$. Since $A_i+B_j$ is an anticanonical divisor, this double cover is a linear projection of $X$ to $\pp^3$ whose image is a hypersurface isomorphic to $\pp^1\times\pp^1$, i.e., defined by a quadric with signature $(2,2)$. The cone over this quadric in $\pp^4$ is one of our singular quadrics. The four divisors $$D_j=L_1+\overline{L_j'}=E_0+l_{0j}$$ on $X_\C$ for $1\leq j\leq 4$ also realize $X_\C$ as a conic bundle $X_\C\to\pp^1_\C$ and for each $j$ we consider the map $X_\C\to\pp^1_\C\times\pp^1_\C$ associated to $D_j$ in the first coordinate and $\overline{D_j}$ on the second. This corresponds to a morphism $f_j:X\to Q$ where $Q\subset\pp^3$ is a hypersurface defined by a quadric of signature $(3,1)$. Since $D_j\cdot \overline{D_j}=2$, this is a double cover, and since $D_j+\overline{D_j}$ is anticanonical, the maps $f_j$ correspond to linear projections of $X$. This shows $(2)$ and $(3)$. We have $$D_j\cdot \overline{D_j}=L_1\cdot\overline{L_1}+L_j'\cdot\overline{L_j'}.$$
 Thus in order to determine whether $f_i:X\to Q$ is real fibered or not, we have to check whether the two intersection points $L_1\cap\overline{L_1}$ and $L_i'\cap\overline{L_i'}$ lie on the same (not real fibered) or different connected components (real fibered) of $X(\R)$. As noted above, both cases occur for exactly two values of $j$.
\end{proof}

With this preparation we are able to determine for which del Pezzo surfaces $X$ there is a real fibered morphism $X\to\pp^2$ whose corresponding linear system is anticanonical.

\begin{prop}\label{prop:pezzohyps}
 Let $X$ be a real del Pezzo surface and $K$ a canonical divisor on $X$. There is a real fibered morphism $f: X\to\pp^2$ such that the pullback of a line is linearly equivalent to $-K$ on $X$ if and only if we have one of the following:
 \begin{enumerate}
  \item $X$ is a quartic surface in $\pp^4$ such that $X(\R)$ is homeomorphic to a disjoint union of two spheres;
  \item $X$ is a cubic hypersurface in $\pp^3$ such that $X(\R)$ is homeomorphic to a disjoint union of a sphere and a real projective plane;
  \item $X$ is a double cover of $\pp^2$ branched along a smooth plane quartic curve $C$ with $C(\R)=\emptyset$ so that $X(\R)$ is homeomorphic to a disjoint union of two real projective planes.
 \end{enumerate}
 In particular, in each case $X(\R)$ has two connected components.
\end{prop}

\begin{proof}
 Let $d=K.K$. For $d=2$ the anticanonical map is a double cover of $\pp^2$ branched along a plane quartic curve $C$, see \Cref{ex:pezzo}. This is real fibered if and only if $C(\R)=\emptyset$ and $X(\R)$ is homeomorphic to a disjoint union of two real projective planes. In general if there exists such a morphism $f$, then $X(\R)$ must be homeomorphic to the disjoint union of $s$ spheres and $r$ real projective planes such that $d=2s+r$ by \cite[Cor.~2.20]{eliich}. Going through the classification of real del Pezzo surfaces in \cite{russo} we see that for $d\neq 2$ this is only possible for a complete intersection of two quadrics in $\pp^4$ such that $X(\R)$ is homeomorphic to a disjoint union of two spheres ($d=4$) or a cubic hypersurface in $\pp^3$ such that $X(\R)$ is homeomorphic to a disjoint union of a sphere and a real projective plane ($d=3$). This shows the ``only if part''. It thus remains to show that the embedded surfaces in $(1)$ and $(2)$ are hyperbolic as these embeddings correspond to the anticanonical linear system, see \Cref{ex:pezzo}. The case $d=3$ is covered by \cite[Prop.~2.2]{victorsurvey}. For the case $d=4$ we can compose a real fibered linear projection $X\to Q$ from \Cref{lem:quarticsurface} to a hyperbolic hypersurface $Q\subset\pp^3$ with the linear projection $Q\to\pp^2$ from a point with respect to which $Q$ is hyperbolic.
\end{proof}

\begin{lem}\label{lem:pezzoram}
 Let $X$ be a del Pezzo surface and $K$ a canonical divisor on $X$. Let $f: X\to\pp^2$ be a real fibered morphism of degree $d$ such that the pullback of a line is linearly equivalent to $-K$ as constructed in the previous Proposition. 
 \begin{enumerate}[a)]
  \item The ramification divisor $R$ is linearly equivalent to $-2K$. 
  \item Let $g$ be a rational function on $X$ whose principal divisor is $R+2K$. Then $g$ has constant sign on each of the two connected components of $X(\R)$ and these signs are not the same.
 \end{enumerate}
\end{lem}

\begin{proof}
 The ramification divisor is linearly equivalent to the canonical divisor on $X$ minus the pullback of the canonical divisor on $\pp^2$. As the latter is $\cO_{\pp^2}(-3)$ its pullback is $3K$. This shows $a)$.
 
 For $d=2$ part $b)$ follows from the fact $X$ is a double cover of $\pp^2$ of the form $y^2=p(x)$ where $p\in\R[x_0,x_1,x_2]_4$ is a globally positive quartic curve. The ramification divisor is given as the zero locus of $y$ whereas $-K$ is the zero set of a linear form $l\in\R[x_0,x_1,x_2]$. Clearly $\frac{y}{l^2}$ has the desired properties.
 
 In the case $d=3$ our surface $X$ is the zero set of a hyperbolic polynomial $h$ and our morphism $f$ is the linear projection from a point $e\in\pp^3$ of hyperbolicity. Its ramification divisor is thus cut out by the interlacer $\textrm{D}_eh$. Again $-K$ is the zero set of a linear form $l\in\R[x_0,x_1,x_2]$ and $\frac{\textrm{D}_eh}{l^2}$ has different sign on the two connected components of $X(\R)$.
 
 In the case $d=4$ we first note that by \Cref{lem:quarticsurface} we can assume (after a linear change of coordinates) that $X$ is cut out by $q_1=x_0^2-(x_1^2+x_2^2+x_3^2)$ and $q_2=x_4^2-p$ for some quadratic form $p\in\R[x_0,x_1,x_2,x_3]_2$ that is nonnegative on $Q=\cV(q_1)\subset\pp^3$. Our real fibered morphism $f:X\to\pp^2$ is then the composition of the two linear projections $X\to Q$ with center $[0:0:0:0:1]$ and $Q\to\pp^2$ with center $[1:0:0:0]$. Thus the ramification locus of $f$ is cut out by $x_0\cdot x_4$. Again $-K$ is the zero set of a linear form $l\in\R[x_0,x_1,x_2]$ and $\frac{x_0\cdot x_4}{l^2}$ has different sign on the two connected components of $X(\R)$.
\end{proof}

Let $X_4$ be a complete intersection of two quadrics in $\pp^4$ such that $X(\R)$ is homeomorphic to a disjoint union of two spheres. We fix a sequence of morphisms \begin{align}X_2\to X_3\to X_4\label{eqn:blowup}\end{align} where each map $f_i:X_{i}\to X_{i+1}$ is the blow up of $X_i$ at real point on a connected component of $X_i(\R)$ that is homeomorphic to a sphere. Further let $E_i$ be the exceptional divisor of $f_i$. By the classification of real del Pezzo surfaces in \cite{russo}, we have that $X_3$ is a cubic hypersurface in $\pp^3$ such that $X(\R)$ is homeomorphic to a disjoint union of a sphere and a real projective plane and $X_2$ is a double cover of $\pp^2$ branched along a smooth plane quartic curve $C$ with $C(\R)=\emptyset$ so that $X(\R)$ is homeomorphic to a disjoint union of two real projective planes. Conversely, every such real del Pezzo surface fits in such a sequence of blow-ups. 

\begin{lem}\label{lem:cubicsurface}
 Consider the cubic hypersurface $X_3\subset\pp^3$.
 \begin{enumerate}[a)]
  \item In addition to $E_3$ there are two more real lines $L$ and $L'$ on $X_3$. These three lines lie on a common plane.
  \item There are two different hyperplanes that contain $L$ and are tangential to a real point on the connected component of $X_3(\R)$ that is homeomorphic to the sphere.
  \item The divisors $H_1$ and $H_2$ on $X_3$ that are defined as the intersections with the hyperplanes from part $b)$ are of the form $H_i=L+L_i+\overline{L_i}$ for some non-real lines $L_i$ on $(X_3)_\C$.
  \item The lines $L_i$ and $\overline{L_i}$ are disjoint from $E_3$. Furthermore, $L_1$ and $L_2$ are disjoint.
  \item Let $f$ be a rational function on $X_3$ whose principal divisor is $L_2+\overline{L_2}-L_1-\overline{L_1}$. Then $f$ has constant sign on each of the two connected components of $X(\R)$ and these signs are not the same.
 \end{enumerate}
\end{lem}

\begin{proof}
 The number of real lines on $X_3$ can be found for example in \cite[p.~302]{russo}. As they all must lie in the component of $X(\R)$ that is homeomorphic to $\R\pp^2$, each two of them intersect in a point. Thus they all lie in a common plane $H_0$ which proves $a)$.
 
 In the affine chart $\R^3=(\pp^3\setminus H_0)(\R)$ the connected component of $X_3(\R)$ that is homeomorphic to a sphere is the boundary of a compact convex set $K\subset\R^3$, namely $K$ is an affine slice of the hyperbolicity cone of the cubic that defines $X_3$. The hyperplanes containing $L$ correspond to a family of parallel affine hyperplanes in $\R^3$. Thus exactly two of them are tangent to $K$. This shows $b)$.
  
 The zero divisors of these hyperplanes $H_i$ contain besides $L$ a plane conic which has an isolated real point, namely the point of tangency. Thus the conic is a complex conjugate pair of lines $L_i$ and $\overline{L_i}$ which shows part $c)$.
 
 In order to show $d)$ assume for the sake of a contradiction that $L_i$ intersects $E_3$. Since $L_i\subset H_i$ and $E_3\subset H_0$, this intersection point must lie on $E_3\cap H_0\cap H_i=E_3\cap L$ which implies that is real. But the only real point of $L_i$ lies on the spherical component of $X_3(\R)$. An analogous argument shows that $L_1$ and $L_2$ are disjoint.
 
 Finally, let $l_i$ be the linear form that cuts out $H_i$. Then by construction $p=l_1l_2$ is an interlacer of the polynomial defining $X_3$. Thus the rational function $f=\frac{p}{l_1^2}$ has constant sign on each of the two connected components of $X(\R)$ and these signs are not the same. Clearly the principal divisor corresponding to $f$ is $L_2+\overline{L_2}-L_1-\overline{L_1}$. Therefore, we have shown part $d)$.
\end{proof}

\begin{figure}[h]
  \begin{center}
    \includegraphics[width=5cm]{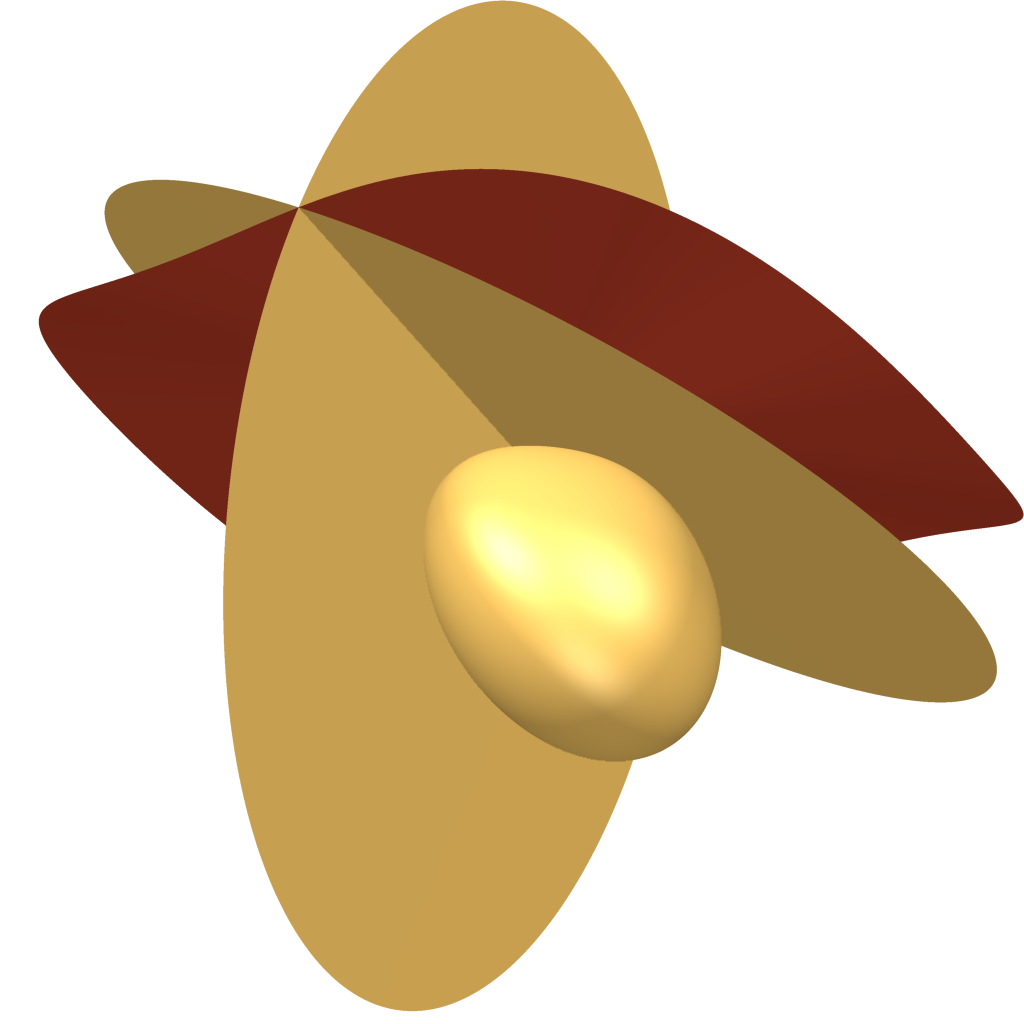}
  \end{center}
  \caption{A cubic hyperbolic hypersurface with two planes that contain a line on the pseudoplane (red) and are tangent to the spherical component (yellow).}
  \label{fig:cubic}
  \end{figure}

\begin{thm}\label{thm:ulrichdelpezzo}
 Let $X$ be a del Pezzo surface and $f: X\to\pp^2$ a real fibered morphism of degree $d$ such that the pullback of a line is the anticanonical divisor class $-K$. Then there is a positive hermitian $f$-Ulrich line bundle.
\end{thm}

\begin{proof}
 We put $X$ into a sequence of blow-ups as in (\ref{eqn:blowup}). Since the lines $L_i$ and $\overline{L_i}$ on $(X_3)_\C$ from \Cref{lem:cubicsurface} are disjoint from $E_3$, they can be identified with some lines on $(X_4)_\C$ which we, by abuse of notation, also denote by $L_i$ and $\overline{L_i}$. The same we do for the proper transforms of $L_i$ and $\overline{L_i}$ in $(X_2)_\C$. We want to apply part $b)$ of \Cref{thm:weil} to the divisor $M=L_2-L_1-K$ where $K$ is a canonical divisor on $X$. Since $X$ is birational to $X_3$ the rational function $f$ from part $e)$ of \Cref{lem:cubicsurface} is also a rational function on $X$ and we have $M+\overline{M}=(f)-2K$. By \Cref{lem:pezzoram} there is a rational function $g$ on $X$ such that $(g)=R+2K$ where $R$ is the ramification divisor. Furthermore, we can choose $g$ to have the same sign on each of $X(\R)$ as $f$. Thus $M+\overline{M}=(\frac{f}{g})+R$ and $\frac{f}{g}$ is nonnegative. It thus remains to show that the dimension $\ell(M)$ of the space of global sections of $\cL(M)$ is at least $d$. To this end we invoke the Theorem of Riemann--Roch for surfaces \cite[Thm.~1.6]{Hart77}:
 $$ \ell(M)+\ell(K-M)=\frac{1}{2}M.(M-K)+1+p_a+s(M)\geq \frac{1}{2}M.(M-K)+1=d. $$Now since the intersection product of $(K-M)$ with the ample divisor $-K$ equals $-2d<0$, it cannot be effective. Thus $\ell(K-M)=0$ and the claim follows.
\end{proof}

We now apply \Cref{thm:ulrichdelpezzo} to the three cases from \Cref{prop:pezzohyps}.
The following consequence is originally due to Buckley and Ko\c{s}ir \cite{buck}.

\begin{cor}
 Every hyperbolic polynomial $h\in\R[x_0,x_1,x_2,x_3]$ of degree three has a definite hermitian determinantal representation.
\end{cor}

\begin{proof}
 First assume that the zero set of $h$ is smooth. Then we are in case $(2)$ of \Cref{prop:pezzohyps} and the claim follows from \Cref{thm:ulrichdelpezzo} and \Cref{prop:ulrichdetrep}. For the singular case note that by \cite{Nuij} the set of all hyperbolic polynomials is the closure of the smooth ones. Further by \cite[Lem.~3.4]{hvelemt} the set of hyperbolic polynomials with a definite hermitian determinantal representation is closed.
\end{proof}

The following consequence is originally due to Hilbert \cite{hilbert1888}.

\begin{cor}
 Every nonnegative ternary quartic is a sum of three squares.
\end{cor}

\begin{proof}
 First consider a nonnegative ternary quartic $p$ with smooth zero set. The hypersurface defined by $y^2-p$ in $\pp(2,1,1,1)$ is an instance of \Cref{prop:pezzohyps}(3). Thus the claim follows from \Cref{thm:ulrichdelpezzo} and \Cref{thm:ulrichsos}. The general case now follows from a limit argument as the set of sums of squares is closed in $\R[x_0,x_1,x_2]_4$.
\end{proof}

\begin{cor}
 The Chow form of a smooth hyperbolic surface in $\pp^4$ of degree four, which is a complete intersection of two quadrics, has a definite hermitian determinantal representation.
\end{cor}

\begin{proof}
 Here we are in case (1) of \Cref{prop:pezzohyps}.
 The claim follows from \Cref{thm:ulrichdelpezzo} together with a straight-forward adaption of the proof of \cite[Thm.~5.7]{eliich} to the hermitian case and \cite[Rem.~4.4]{eliich}.
\end{proof}

\begin{rem}
 We have seen that every nonnegative polynomial $p\in\R[x_0,x_1,x_2]_4$ is a sum of squares and every hyperbolic polynomial $h\in\R[x_0,x_1,x_2,x_3]_3$ has a definite hermitian determinantal representation, i.e., the associated real fibered morphisms admit a positive Ulrich sheaf. This is no longer true if we increase the degrees: Not every nonnegative polynomial $p\in\R[x_0,x_1,x_2]_6$ is a sum of squares \cite{hilbert1888} and there are hyperbolic polynomials $h\in\R[x_0,x_1,x_2,x_3]_4$ such that no power $h^r$ has a definite determinantal representation, take for example the polynomial considered in \cite{kumvam}. Double covers of $\pp^2$ ramified along plane sextic curves and quartic hypersurfaces in $\pp^3$ both belong to the class of K3 surfaces. So it would be very interesting to understand which real fibered morphisms $X\to\pp^2$ from a K3 surface $X$ admit a positive Ulrich bundle. Note that (not necessarily positive) Ulrich sheaves of rank two on K3 surfaces have been constructed in \cite{k3ulrich}. Similarly, we can increase the dimensions: Not every nonnegative polynomial $p\in\R[x_0,x_1,x_2,x_3]_4$ is a sum of squares \cite{hilbert1888} and it is not known whether there are hyperbolic polynomials $h\in\R[x_0,x_1,x_2,x_3,x_4]_3$ such that no power $h^r$ has a definite determinantal representation, see \cite[\S5]{sos42} for cubic hyperbolic hypersurfaces. Double covers of $\pp^3$ ramified along quartic surfaces and cubic hypersurfaces in $\pp^4$ both belong to the class of  Fano threefolds of  index two. In \cite[\S6]{ulrichintro}  Ulrich sheaves of rank two on such threefolds have been constructed.
\end{rem}

\begin{prob}
 Understand which finite surjective and real fibered morphisms $X\to\pp^n$ admit a positive Ulrich sheaf for $X$ a K3 surface or a Fano threefold of index two. If such exist, what are their ranks? Are there hyperbolic cubic hypersurfaces in $\pp^4$ that do not carry a positive Ulrich sheaf?
\end{prob}

We conclude this section with some examples.

\begin{ex}\label{ex:cubic}
 Let $h=x_0^3-x_0(2x_1^2+2x_2^2+x_3^2)+x_1^3+x_1x_2^2$. The hypersurface $X=\cV(h)\subset\pp^3$ is hyperbolic with respect to $e=[1:0:0:0]$ and contains the real line $L=\cV(x_0,x_1)$. The hyperplanes $H_1=\cV(x_0)$ and $H_2=\cV(x_0-x_1)$ contain $L$ and are tangent to the hyperbolicity cone of $h$. The quadratic polynomial $p=x_0(x_0-x_1)$ is an interlacer of $h$ and its zero divisor on $X$ is $$D=2L+L_1+\overline{L_1}+L_2+\overline{L_2}$$ where $L_1=\cV(x_0,x_1+\textrm{i}x_2)$ and $L_2=\cV(x_0-x_1,x_2+\textnormal{i}x_3)$. Thus we have $D=M+\overline{M}$ with $M=L+L_1+L_2$. The space of all quadrics vanishing on $L, L_1$ and $L_2$ is spanned by $$x_0(x_0-x_1), x_0(x_2+\textnormal{i}x_3),(x_0-x_1)(x_1+\textrm{i}x_2).$$The minimal free resolution over $S=\R[x_0,x_1,x_2,x_3]$ of the ideal in $S/(h)$ generated by these quadrics has length one and is given by the matrix $$\begin{pmatrix}x_0+x_1& -x_2-\textnormal{i}x_3& -x_1-\textnormal{i}x_2\\ -x_2+\textnormal{i}x_3& x_0-x_1& 0\\ -x_1+\textnormal{i}x_2& 0&
      x_0                                                                                                                                                                                                                                                                                                                                                                                                                                                                                                                                                                                                                                                                                                                                                                                                                                                                                                        \end{pmatrix}
 .$$This matrix is hermitian and positive definite at $e$. Its determinant is indeed $h$.
\end{ex}

\begin{ex}
 Consider the following nonnegative ternary quartic {$$p=x_0^4+2 x_0^2 x_1^2+2 x_0 x_1^3+x_1^4+x_0^2 x_2^2-2 x_0 x_1 x_2^2-x_1^2 x_2^2+x_2^4$$}\normalsize and  let $X\subset\pp(2,1,1,1)$ be the corresponding double cover defined by $y^2=p$. On $X$ we have the two lines $$L_1=\cV(y+(1-\textnormal{i})x_1^2-x_2^2,x_0+\textnormal{i}x_1)\textrm{ and }L_2=\cV(y+\textnormal{i}x_1x_2-x_1^2+x_2^2,x_0+\textnormal{i}x_2).$$
 The principal divisor associated to the rational function $$f=\frac{x_0^2+x_2^2}{y-x_0x_1-x_1^2+x_2^2}$$ is $L_2-L_1+\overline{L_2}-\overline{L_1}$. As $L_1$ and $L_2$ are both non-real lines, this implies that $f$ has constant sign on each of the two connected components of $X(\R)$. Evaluating $f$ at points from the two different components, for example at $$[y:x_0:x_1:x_2]=[\pm1:1:0:0],$$ shows that $f$ changes sign.  
 Letting $K$ be a canonical divisor on $X$, we have for $M=L_2-L_1-K$ that $M+\overline{M}=(f)-2K$ as in the proof of \Cref{thm:ulrichdelpezzo}.
 We can realize the divisor class of $M$ by the ideal in $\R[y,x_0,x_1,x_2]/(y^2-p)$ that is generated by $y-x_0x_1-x_1^2+x_2^2$ and $(x_0+\textnormal{i}x_1)(x_0+\textnormal{i}x_2)$. The minimal free resolution over the ring $\R[y,x_0,x_1,x_2]$ has length one and is given by the matrix
 $$\begin{pmatrix}
    y-x_0x_1-x_1^2+x_2^2 & x_0^2-x_1x_2+\textrm{i}(x_0x_1+x_0x_2) \\
    x_0^2-x_1x_2-\textrm{i}(x_0x_1+x_0x_2) & y+x_0x_1+x_1^2-x_2^2
   \end{pmatrix}.$$Indeed, we have that $A^2=p\cdot I$. Therefore $$p= (x_0x_1+x_1^2-x_2^2)^2 + (x_0^2-x_1x_2)^2+(x_0x_1+x_0x_2)^2.$$
 The key ingredient for this construction was the rational function $f$. One way to find it is to blow down a real line of $X$ and then proceed as in \Cref{lem:cubicsurface}.
\end{ex}

\begin{ex}
 The del Pezzo surface $$X=\cV(x_0^2+x_1^2+x_2^2-x_3^2,x_0^2+4x_1^2+9x_2^2-x_4^2)\subset\pp^4$$ of degree four is hyperbolic with respect to the line $E$ spanned by $[0:0:0:1:0]$ and $[0:0:0:0:1]$. The rational function $\frac{(x_3-x_0)(x_4-x_0)}{x_0^2}$ has different signs on the two connected components of $X(\R)$. Its corresponding divisor is of the form $M+\overline{M}+2K$ for a suitable divisor $M$ on $X_\C$ with $\ell(M)=4$. Thus $\cL(M)$ is a positive hermitian Ulrich line bundle by \Cref{thm:weil}. As in \cite[Thm.~0.3]{ES03} we obtain from that the following determinantal representation of the Chow form of $X$, written in the Pl\"ucker coordinates:  
    \[\scalebox{0.6}{$\begin{pmatrix}
    2 x_{03}-2 x_{04}+2 x_{34}& 4 x_{01}-6 \textrm{i} x_{02}+4 x_{13}-6 \textrm{i} x_{23}&      -2 x_{01}+2 \textrm{i} x_{02}-2 x_{14}+2 \textrm{i} x_{24}& -2 \textrm{i} x_{12}\\4 x_{01}+6 \textrm{i} x_{02}+4 x_{13}+6 \textrm{i} x_{23}& -2 x_{03}-2 x_{04}+2 x_{34}&      10 \textrm{i} x_{12}& 2 x_{01}-2 \textrm{i} x_{02}-2 x_{14}+2 \textrm{i} x_{24}\\-2 x_{01}-2 \textrm{i} x_{02}-2 x_{14}-2 \textrm{i} x_{24}& -10 \textrm{i} x_{12}&      2 x_{03}+2 x_{04}+2 x_{34}& -4 x_{01}+6 \textrm{i} x_{02}+4 x_{13}-6 \textrm{i} x_{23}\\2 \textrm{i} x_{12}& 2 x_{01}+2 \textrm{i} x_{02}-2 x_{14}-2 \textrm{i} x_{24}&    -  4 x_{01}-6 \textrm{i} x_{02}+4 x_{13}+6 \textrm{i} x_{23}& -2 x_{03}+2 x_{04}+2 x_{34}
   \end{pmatrix}$}\]\normalsize
 We observe that it is hermitian and positive definite at $E$.
\end{ex}

%\bigskip

% \noindent \textbf{Acknowledgements.}
% We would like to thank somebody.

 \bibliographystyle{alpha}
 \bibliography{biblio}
 \end{document}

%% file: operators.tex
\usepackage{mathrsfs}
\newcommand\ratto{\dashrightarrow}

\newcommand{\Hom}{\operatorname{Hom}} %External Hom
\newcommand{\shHom}{\operatorname{{\mathscr Hom}}} %Sheaf Hom
\newcommand{\Spec}{\operatorname{Spec}}

\newcommand{\Supp}{\operatorname{Supp}}

\newcommand{\rank}{\operatorname{rank}} %rank of a free module
\newcommand{\Quot}{\operatorname{Quot}}

\newcommand{\tr}{\operatorname{tr}}

%% file: fonts.tex
\newcommand{\cD}{{\mathcal D}}
\newcommand{\cE}{{\mathcal E}}
\newcommand{\cF}{{\mathcal F}}
\newcommand{\cG}{{\mathcal G}}

\newcommand{\cK}{{\mathcal K}}
\newcommand{\cL}{{\mathcal L}}

\newcommand{\cO}{{\mathcal O}}
\newcommand{\cP}{{\mathcal P}}

\newcommand{\cV}{{\mathcal V}}

\newcommand{\fm}{{\mathfrak m}}

\newcommand{\fp}{{\mathfrak p}}

\newcommand{\C}{{\mathbb C}}
\newcommand{\R}{{\mathbb R}}
\newcommand{\pp}{\mathbb{P}}

\newcommand{\N}{{\mathbb N}}
\newcommand{\Z}{{\mathbb Z}}